\documentclass[reqno]{amsart}
\usepackage[english]{babel}
\usepackage{amscd,amssymb,amsmath,amsfonts,latexsym,amsthm}
\usepackage{inputenc}
\usepackage{graphicx,color}
\newtheorem{thm}{Theorem}[section]
\newtheorem{cor}[thm]{Corollary}
\newtheorem{lem}[thm]{Lemma}
\newtheorem{prop}[thm]{Proposition}
\newtheorem{defn}[thm]{Definition}
\newtheorem{example}[thm]{Example}

\numberwithin{equation}{section}
\begin{document}
%\title[Triviality of the second cohomology groups]{Triviality of the second cohomology groups of residually solvable Lie and Leibniz algebras.}

\title[Residually solvable extensions of an infinite dimensional filiform Leibniz algebra]{Residually solvable extensions of an infinite dimensional filiform Leibniz algebra}

\author{K.K. Abdurasulov$^{1,3}$, B.A. Omirov$^{1,3}$, I.S. Rakhimov$^{2,3}$, G.O. Solijanova$^{3}$}

\address{$^1$ Institute of Mathematics Uzbekistan Academy of Sciences, Tashkent, Uzbekistan} \email{abdurasulov0505@mail.ru, omirovb@mail.ru}

\address{$^2$ Faculty of Computer and Mathematical Sciences, Universiti Teknologi MARA (UiTM), Shah Alam, Malaysia}
\email{isamiddin@uitm.edu.my}

\address{$^3$ National University of Uzbekistan, Tashkent, Uzbekistan} \email{gulhayo.solijonova@mail.ru}

\begin{abstract} In the paper we describe the class of all solvable extensions of an infinite-dimensional filiform Leibniz algebra. The filiform Leibniz algebra is taken as a maximal pro-nilpotent ideal of a residually solvable Leibniz algebra. It is proved that the second cohomology group of the extension is trivial.
\end{abstract}



\maketitle
\normalsize

\section{Introduction}

Lie algebras play an essential role in Mathematics and Physics. Their structure theory has been intensely studied for many years by mathematicians and theoretical physicists. In this paper we will be treating infinite-dimensional Lie and Leibniz algebras. The motivation comes from the results with applications of the solvable extension method to finite-dimensional Lie algebras obtained earlier (see Section \ref{S11}) and finite-dimensional Leibniz algebras obtained recently (see Section \ref{S13}).  All algebras considered in the paper are supposed to be over the field of complex numbers $\mathbb{C}$ unless otherwise specified.
\subsection{\textbf{Finite-dimensional Lie algebras.}} \label{S11}

It is well-known that the fundamental Levi theorem stating that every finite-dimensional Lie algebra is isomorphic to the semidirect sum of a semisimple Lie algebra and the maximal solvable ideal (the radical) greatly simplifies understanding the structure of a given Lie algebra.  The classification of simple and semisimple Lie algebras were obtained by W. Killing and \'{E}. Cartan using root systems. Simple Lie algebras are given as five exceptional algebras $E_6$, $E_7$, $E_8$, $F_4$, $G_2$ and series  $A_n$, $B_n$, $C_n$ and $D_n$. While a complete
classification of solvable and nilpotent Lie algebras seems not to be feasible. As for solvable Lie algebras, then various methods of their classification in low-dimensional cases have been implemented and lists of isomorphism classes were presented. Here are just a few words on one of these methods closely related to that problems studied in this paper. This is a method originated by V.V. Morozov \cite{Morozov} on construction of solvable Lie algebras by their nilradicals. It turned out what was noticed by V.V. Morozov and C.M. Mubarakzjanov that there is interrelations between a few invariants of Lie algebras: the dimension and the number of generators of the nilradical, the co-dimension of the nilradical, the number of nil-independent derivations, the existence of inner and outer derivations, the dimensions of the first and second (co)homology spaces. One of such kind relations states that the co-dimension of the nilradical of the Lie algebra is at most the number its nill-independent derivations. The fact has been used to construct the solvable Lie algebras in \cite{Mubarakzjanov1, Mubarakzjanov2, Mubarakzjanov3, Mubarakzjanov4}. By using these relationships the classification of solvable extensions has been given for the following classes of nilpotent Lie algebras in low-dimensions with Abelian \cite{Ndogmo1, Ndogmo2}, Heisenberg \cite{Rubin}, Borel \cite{Snobl3}, $\mathbb{N}$-graded \cite{Campoamor}, filiform and quasifiliform nilradicals \cite{Snobl1, Snobl2} (also see \cite{Snobl} and references therein). A few other nilradicals cases are given in \cite{Shabanskaya1, Shabanskaya2, Shabanskaya3}. In the finite-dimensional case, the concept of filiform Lie algebra has been introduced by Vergne in \cite{Vergne}. It turned out that the case of filiform and quasi-filiform nilradicals the relationships mentioned above much simplify the situation. Moreover, there are cases when the co-dimension of the nilradical is equal to the number of its generators. In these cases the structure of solvable Lie algebras is quite rigid. It was observed that a such Lie algebra is unique up to isomorphism (see \cite{Khal}), its center is trivial and all derivations are inner. In addition, often such Lie algebras have trivial second Chevalley cohomology groups (see \cite{Ancochea1, Leger}).

%Now we have two ways to generalize the Lie algebras, the first is in terms of the dimension the another one is in terms of the structure. Let us consider the former.

\subsection{\textbf{Infinite-dimensional Lie algebras.}}
As mentioned above the problem of classification of simple finite-dimensional Lie algebras
over the field of complex numbers was solved by the end of the 19th century
by W. Killing and \'{E}. Cartan. And just over a decade later, \'{E}. Cartan classified simple
infinite-dimensional Lie algebras of vector fields on a finite-dimensional
space. Then B. Weisfeiler \cite{Weisfeiler}  finds an algebraic proof of Cartan's classification theorem reducing the problem to the classification of simple $\mathbb{Z}$-graded Lie algebras of finite ``depth''.

At the present time there is no general theory of infinite-dimensional Lie algebras. There are, however, four classes of infinite-dimensional Lie algebras that underwent a more or less intensive study due to their various applications mostly in Physics. These are, first of all, the above-mentioned Lie algebras of vector fields, the second class consists of Lie algebras of smooth mappings of a given manifold into a finite-dimensional Lie algebra, the third class is the classical Lie algebras of operators in a Hilbert or Banach space and finally, the fourth class of infinite-dimensional Lie algebras is the class of the so-called Kac-Moody algebras.

 In the classification theory of infinite-dimensional Lie algebras, several deep results were obtained with Galois cohomology methods exhibiting exciting connections between forms of multi-loop algebras and the Galois theory of forms of algebras over rings. This branch of structure theory is complemented by the connection between the classification of generalized Kac-Moody algebras and automorphic forms.
 At the same time, the study of some classes of infinite-dimensional Lie algebras is still in its infancy. Even such concepts as solvability and nilpotency, as well as their interrelations, are not completely understood.

It is known that there is no analogue of the Levi decomposition in infinite-dimensional Lie algebras. Moreover, the analogue of Engel's theorem for infinite-dimensional Lie algebras is also wrong. Only in the 1990s E.Zelmanov could solve the Burnside problem, which connects the nilpotency property of a Lie algebra with adjoint operators satisfying the $n$-Engel's condition \cite{Zel'manov}.

Note that there are some examples of the so-called pro-solvable Lie algebras whose maximal pro-nilpotent ideal is $\mathbb{N}$-graded Lie algebra of maximal class (infinite-dimensional filiform Lie algebra) the method described above for finite-dimensional solvable Lie algebras by means of its nilradical is applicable (see \cite{Qobil2}). It should be noted that in all of these mentioned examples, the codimension of the maximal pro-nilpotent ideal of a pro-solvable algebra coincides with the number of generators of the pro-nilpotent ideal.
 %{\color{red}{The structure of solvable Lie algebras such that the dimension of the complementary subspace to the nilradical is equal to the number of generators of the nilradical is quite rigid. Such algebras are unique up to isomorphism. Besides, they have many excellent properties: their center is zero, and any derivation of such algebra is inner.
%
%In addition, often such Lie algebras have trivial second cohomology groups.}}
%%Therefore, one of the main goals of the work is to determine, in the infinite-dimensional case, analogues of the concepts of solvable and nilpotent to which the method of constructing solvable finite-dimensional Lie and Leibniz algebras extends in terms of the structure of the nilradical and their non-nilpotent derivations.

The infinite-dimensional analogue of filiform Lie algebras has been introduced by A.Fialowski a long time ago in \cite{Fialowski}.
Nevertheless, the systematic study of infinite-dimensional cases has not been of interest. The attempts were occasional depending mainly on some applications in Physics and Geometry. For instance, in \cite{Khak}, two classes of infinite-dimensional Lie algebras called potentially nilpotent and potentially solvable were introduced in connection with the study of their deformations.

\subsection{\textbf{Finite-dimensional Leibniz algebras.}} \label{S13}

Many results of the theory of Lie algebras have been extended to the case of Leibniz algebras. For instance, the classical results on Cartan subalgebras, Levi's decomposition, properties of solvable algebras with a given nilradical and others from the theory of Lie algebras are also true for Leibniz algebras.
D. Barnes \cite{Barnes2} has proved an analogue of Levi Theorem for Leibniz algebras. It was shown that a Leibniz algebra is decomposed into a semidirect sum of its solvable radical and a semisimple Lie algebra. Therefore, the description of finite-dimensional Leibniz algebras is reduced to the study of solvable Leibniz algebras. D.Barnes also noted that the non-uniqueness of the semisimple subalgebra $S$ appeared in Levi-Malcev theorem (the minimum dimension of Leibniz algebra where this phenomena appears is six). It is known that in the case of Lie algebras the semisimple Levi quotient is unique up to a conjugation via an inner automorphism. However, the conjugacy property in the case of Leibniz algebras is not true, in general. This phenomenon has been studied in \cite{KuLO}. The authors give the conditions for the semisimple part of Leibniz algebras in the decomposition to be conjugated.

The generalizations of the solvable extension method mentioned in Section \ref{S11} to some cases of Leibniz algebras have been given in \cite{Casas, Abror1, Karimjanov, Abror2, Abror3}. The construction of solvable extensions has been treated for the following
classes of Leibniz algebras: with Abelian \cite{Rustam2,Abror3}, Heisenberg \cite{Bosko}, filiform and null-filiform nilradicals \cite{Casas,Omirov1}, naturally graded filiform \cite{Casas1, Ladra}, naturally graded quasi-filiform \cite{O'ktam}, the direct sum of null-filiform \cite{Abror2, Abror3}  (also see \cite{Shabanskaya4} and references therein). Note that in the paper \cite{O'ktam}, the construction of a finite dimensional complete Leibniz algebra with the subspace complementary space to the nilradical is not maximal was given. It also was proven there that the second (co)homology group of these type of Leibniz algebras is trivial.

%While in the paper \cite{O'ktam}, the completeness and  rigidity of a solvable Leibniz algebra with naturally graded filiform nilradical and the dimension of the complementary subspace to the nilradical is less than the dimension of  the rank was showed.
%
%While in the paper \cite{O'ktam}, it is showed that one of the solvable Leibniz algebras with naturally graded filiform nilradical and the dimension of the complementary subspace to the nilradical is less than the dimension of  the rank is complete and  rigid.

\subsection{\textbf{Infinite-dimensional Leibniz algebras.}}

%This paper is primarily devoted to an effort to understand the analogue of the results above to infinite-dimensional Leibniz algebras case.

  An extensive study of Lie and Leibniz algebras leaded many beautiful results and generalizations.
However, in spite of great interest and applications the study of some classes of infinite-dimensional Leibniz algebras is still in its infancy. Such in the case of infinite-dimensional Leibniz algebras the concepts of solvability and nilpotency, as well as their interrelations, are not yet completely understood.

 For most classes of the infinite-dimensional Leibniz algebras an analogue of the Levi decomposition and Engel's theorem also do not hold.
 %; for example, affine Lie algebras have a radical consisting of their center, but cannot be written as a semidirect product of its center and another Lie algebra.
In this paper we are going to treat infinite-dimensional Lie and Leibniz algebras around the solvable extension method narrated in Sections \ref{S11} and \ref{S13}. This motivates to introduce a reasonable analogue of the concepts of solvability and nilpotency in the infinite-dimensional cases.
  %So the paper focuses on infinite-dimensional analogues of the concepts of nilpotency and solvability of Lie and Leibniz algebras,
  They have been called pro-nilpotent and pro-solvable algebras.
   %and study their properties.
   The concept of pro-nilpotency is defined by two conditions: the intersection of all members of the lower central series is trivial (so-called residually nilpotency property), and the quotient algebra by any member of the central series is finite-dimensional (this condition allows one to endow an infinite-dimensional pro-nilpotent algebra with the topology of the inverse limit of finite-dimensional spaces). One can call ``tends to zero'' such behaviour of an ideal only with a fair amount of fantasy, for instance, a free finitely generated Leibniz algebra is also pro-nilpotent. Nevertheless, the condition of residual nilpotency is such that finite-dimensional nilpotent Leibniz algebras are also pro-nilpotent. By analogue with solvable and pro-nilpotent Lie (Leibniz) algebras, we consider pro-solvable algebras. %Namely, pro-solvable Lie (Leibniz) algebras are defined by two similar properties: the intersection of all members of the derived series is equal to zero, and the quotient algebra by any member of the derived series is finite-dimensional.

Note that there exist some examples of pro-solvable Lie algebras whose maximal pro-nilpotent ideal is $\mathbb{N}$-graded Lie algebra of maximal class (infinite-dimensional filiform Lie algebra) for which the method of describing finite-dimensional solvable Lie algebras by means of its nilradical is completely agrees. It should be noted that in this case the phenomena occurred in finite-dimensional case, where the dimension of complementary subspace equals the number of generators of the nilradical are valid. That is the structure of infinite-dimensional solvable algebras such that the dimension of the complementary subspace to the nilradical is equal to the number of generators of the nilradical is quite rigid. Such algebras are unique up to isomorphism, their center is zero, and all the derivations are inner.
In addition, often such infinite-dimensional algebras have trivial second cohomology groups. Therefore, one of the main goals of the paper is: to determine, in the infinite-dimensional case, analogues of the concepts of solvability and nilpotency; then to extend the method of constructing solvable finite-dimensional Lie and Leibniz algebras in terms of the structure of the nilradical and their nill-independent derivations. Remind that the infinite-dimensional analogue of finite-dimensional filiform Leibniz algebras has been  introduced in \cite{Omirov}. The author called them \emph{thin Leibniz algebras}.

% But in infinite-dimensional cases, there is no definite method for constructing residually solvable algebras.
%{\color{blue}{In this work, we apply these approaches for the study of residually solvable Leibniz algebras by fixing their maximal potentally nilpotent ideals.}}

%{\color{red}{In \cite{Rustam}, it is proved that any solvable Lie algebra with the complementary subspace to the nilradical being of less dimension than the rank, admits an outer derivation. While in the paper \cite{O'ktam}, the completeness and  rigidity of the solvable Leibniz algebra with naturally graded filiform nilradical and the dimension of the complementary subspace to the nilradical is less than the dimension of  the rank were shown.}}
%{\color{blue}{Our goal in this article is to study what these results will be in infinite-dimensional Lie and Leibniz algebras.}}

%{\color{red}{ While in the paper \cite{O'ktam}, it is showed that one of the solvable Leibniz algebras with naturally graded filiform nilradical and the dimension of the complementary subspace to the nilradical is less than the dimension of  the rank is complete and  rigid.}}
%{\color{blue}{Our goal in this article is to study what these results will be in infinite-dimensional Lie and Leibniz algebras.}}

The organization of the paper is as follows. In Section \ref{Pre} we recall some theorems on finite-dimensional Lie and Leibniz algebras, definitions needed and some important properties of infinite-dimensional Lie and Leibniz algebras. We start Section \ref{main} with studying infinite-dimensional Lie algebras.
%Having analyzed the structure of solvable finite-dimensional Lie algebras with a given nilradical and the maximal complementary subspace to that, we wonder if this structure can be extended to the infinite-dimensional case. To achieve our goal, we focus our research on an infinite-dimensional analogue of nilpotent and solvable Lie and Leibniz algebras of the so-called pro-nilpotent and pro-solvable Lie and Leibniz algebras. Recall that pro-nilpotent Lie (Leibniz) algebras are defined by two properties: the intersection of all members of the lower central series is zero (so-called residually nilpotent property), and the quotient algebra by any member of the central series is finite-dimensional (this condition allows us to endow an infinite-dimensional pro-nilpotent Leibniz algebra with the topology of the inverse limit of finite-dimensional spaces). One can call "tends to zero" such behaviour of an ideal only with a fair amount of fantasy, for instance, a free finitely generated Leibniz algebra is also pro-nilpotent. Nevertheless, the condition of residual nilpotency is such that finite-dimensional nilpotent Leibniz algebras are also pro-nilpotent.
% By analogy with solvable and pro-nilpotent Lie (Leibniz) algebras, we consider pro-solvable algebras.
 All residually solvable extensions of the family whose the dimension of the complementary space to maximal pro-nilpotent ideal is less than maximal is classified and  it is showed that all the obtained algebras have outer derivations.
 %This can be example that does not contradict for the main result in the paper \cite{Rustam}. From the examples we know, we can hypothesize that the main theorem from \cite{Rustam} is also true for infinite dimensions.
 The completeness of maximal residually solvable Lie algebra with maximal pro-nilpotent ideal $\mathfrak{m}_0$ are and trivialness of their second cohomology groups are proven.
Section \ref{sub2} is devoted to the study of infinite-dimensional Leibniz algebras. All residually solvable extensions whose maximal pro-nilpotent ideal is an infinite-dimensional analogue of a filiform Leibniz algebra  are classified. The triviality of the first and second cohomology groups of the family of the maximal residually solvable extensions are proven.
Here one of the differences among the others between Lie and Leibniz algebras appears. Furthermore, in this subsection we provide an example of residually solvable  complete Leibniz algebra with not maximal co-dimension (see Theorem \ref{t2}).

%{\color{red}{This gives one of the different properties of a finite dimensional Lie and Leibniz algebras. In this subsection we show a complete residually solvable Leibniz algebra which the dimension of the complementary space is not maximal.}} {\color{blue} {(Tushunarsiz)}}

The infinite-dimensional algebras considered are supposed to have a countable basis such that any element of the algebra is represented as a finite linear combination of the basis elements.
%{\color{green}{Furthermore, throughout the paper any residually solvable Leibniz algebra whose maximal by inclusion pro-nilpotent ideal is $N$ and the codimension of  $N$ is  $n$ is denoted by  $R(N,n)$.}}

 \normalsize
\section{Preliminaries}\label{Pre}
The purpose of this section is to establish the notation, clarify the definitions used in the following text and to give preliminary results.

%In this section we give necessary definitions and preliminary results.
\begin{defn} An algebra $L$ over a field $\mathbb{F}$ is said to be a Leibniz algebra if the following identity
$$[x, [y, z]] = [[x, y], z]-[[x, z], y]$$
holds true for any $x, y, z \in L$ , where $[\cdot,\cdot]$ stands for the product in $L$.
\end{defn}
We introduce the notation $Leib(a,b,c)=[a,[b,c]]-[[a,b],c]+[[a,c],b]$ to use it later.

If $[x, x] = 0$ then $Leib(x,y,z)=0$ is the Jacobi
identity. Therefore, a Leibniz algebra is a ``noncommutative'' analogue of a Lie algebra.

For a given Leibniz algebra $L$ we define the following two-sided ideals
\begin{itemize}
 % \item {$Leib(a,b,c)=[a,[b,c]]-[[a,b],c]+[[a,c],b]$,}
  \item {${\rm Ann}_r(L) =\{x \in L \mid [y,x] = 0,\ \text{for \ all}\ y \in L \},$}
  \item {${\rm Center}(L) =\{x \in L \mid [x,y]=[y,x] = 0,\ \text{for \ all}\ y \in L \}$.}
\end{itemize}
called the {\it right annihilator} and the {\it center} of $L$, respectively.

\begin{defn} A linear map $d \colon L \rightarrow L$ is said to be a
 {derivation} if for all $x, y \in L$, the following derivation rule \[d([x,y])=[d(x),y] + [x, d(y)] \, \] holds true.
\end{defn}
The set of all derivations of $L$ is a Lie algebra with respect to the composition. It is denoted by $\mathrm{Der}(L).$ For $x \in L$, as usual, $ad_x$ denotes the map $ad_x: L \rightarrow L$ defined by $ad_x(y)=[x,y], \ \forall y \in L.$
Obviously, $ad_x$ is a derivation called {\it inner derivations}. The set of all inner derivations of $L$ is denoted by $Inner(L).$ No inner derivations in $\mathrm{Der}(L)$ are called {\it outer derivations}. The set $Inner(L)$ is an ideal of $\mathrm{Der}(L)$.

\begin{defn} A Leibniz algebra $L$ is said to be  {complete} if ${\rm Center}(L)=0$ and all its derivations are
inner.
%\emph{\cite{Ancochea}}.
\end{defn}
Let $R$ be a finite dimensional solvable Leibniz algebra with nilradical $N$ and $Q$ be the subspace complementary to the nilradical $N$.
We denote by $N_{max}$ the class of all algebras with a property that $\dim Q = \dim (N/N^2)$ for some solvable Leibniz algebra $R$. In \label{R}\cite{Rustam} the following theorem was proven.

%$N_{max} = \{\text{nilpotent \ algebra} \ N \ |\ \text{there\ exists\ solvable\ Lie\ algebra}\ R \ \text{with\
%nilradical}\ N \ \text{such}\\ \text{ that}\ \dim Q = \dim (N/N^2)\}$.
%In the paper \cite{Rustam}, the following theorem is proved
\begin{thm} Let $R' = N \oplus Q'$ be a solvable Leibniz algebra such that $N\in N_{max}$ and $\dim Q' < \dim (N/N^2)$. Then $R'$ admits an outer derivation.\end{thm}
%In \cite{Qobil1}, the following theorem is proved.
%\begin{thm}\label{q1}\cite{Qobil1}
%Let $R=N\oplus Q$ be a solvable Leibniz algebra, where $\dim Q=\dim(N/N^2)$. Then $R$ is complete.
%\end{thm}

Let $L$ be an infinite-dimensional Lie (respectively, Leibniz) algebra with countable basis.
As usual
$$L^1=L, \quad L^{k+1}=[L^k,L],  \ k \geq 1; \qquad L^{[1]}=L, \quad L^{[s+1]}=[L^{[s]},L^{[s]}], \ s \geq 1.$$
are the {\it lower central} and the {\it derived series}, respectively.
We imitate a definition given in \cite{Mil2008} from Lie to Leibniz algebras as follows.

\begin{defn}\label{def1}  A Leibniz algebra $L$ is called residually nilpotent $($respectively, solvable$)$ if $\bigcap\limits_{i=1} ^{\infty}L^{i}=0$ $(respectively,\ \bigcap\limits_{i=1} ^{\infty}L^{[i]}=0)$.
\end{defn}
%Consequently, we can give the definition of the notion of residually nilpotent Leibniz algebras as follow:
%\begin{defn}\label{def1}  A Leibniz algebra $L$ is called residually nilpotent if $\bigcap\limits_{i=1} ^{\infty}L^{i}=0$.
%\end{defn}
Here are some examples.
\begin{example}\label{exam1} Lie algebra $L$ given on a basis $\{e_0,e_1,e_2,\dots\}$ by
%is basis of a Lie algebra $L_1$ and for all $i\geq 2$
%{\color{red}{$$L:\left\{\begin{array}{lll}
% [e_0,e_1]=e_{0},\\[1mm]
% [e_0,e_i]=e_{i-1},& i\geq3,\end{array}\right.$$-\text{algebrani ko'paytirish jadvalida kamchilik bilan ketib qolgan ekan}}}
$$L:\left\{\begin{array}{lll}
  [e_0,e_i]=e_{i-1},& i\geq3,\end{array}\right.$$

% holds.
 is residually solvable but not residually nilpotent.
\end{example}

\begin{example}\label{exam2} Let $L$ be a Lie algebra given on a basis $\{e_0,e_1,e_2,\dots\}$ by the following commutation relations
%Let $\{e_0,e_1,e_2,\dots\}$ is basis of a Lie algebra $L_2$ and for all $i,j\geq 1$
$$L: [e_i,e_j]=e_0.$$ Then $L$ is both residually nilpotent and residually solvable.
\end{example}

The concepts of potentially nilpotency and solvability also can be extended from Lie algebras case given in \cite{Khak} to Leibniz algebras as follows.
%And the following definition was used in the paper \cite{Khak} for the Lie algebra case. We take the definition for the Leibniz case also.

\begin{defn}\label{def2} A Leibniz algebra $L$ is said to be potentially nilpotent $($respectively, solvable$)$, if $\bigcap\limits_{i=1} ^{\infty}L^{i}=0$ $($respectively, $\bigcap\limits_{i=1}^{\infty}L^{[i]}=0)$ and $dim (L^i/L^{i+1}) < \infty$ $($respectively, $dim (L^{[i]}/L^{[i+1]}) < \infty)$ for any $i\geq 1$.
\end{defn}

An infinite-dimensional Leibniz algebra $F$ given on a basis $\{e_1,e_2,\dots\}$ by
$[e_i,e_1]=e_{i+1} \ \ i\geq 2$ is potentially nilpotent but not potentially solvable.
%\begin{example}\label{exam3} Let $F$ be an algebra given on a basis $\{e_1,e_2,\dots\}$ by
%$$F: [e_i,e_1]=e_{i+1} \ \ i\geq 2.$$
%\end{example}

The definitions of pro-nilpotency and pro-solvability of algebras have been introduced in \cite{Mil2019} as follows.
\begin{defn} \label{def3} An algebra $L$ is called pro-nilpotent $($respectively, pro-solvable$)$, if $\bigcap_{i=1} ^{\infty}L^{i}=0$ and $dim (L/L^{i}) < \infty$ $($respectively, if $\bigcap\limits_{i=1} ^{\infty}L^{[i]}=0$ and $dim (L/L^{[i]}) < \infty)$ for any $i\geq 1$.
\end{defn}

Observe that there is the following isomorphism of vector spaces:
$$L/L^{i} \cong L/L^{2}\oplus L^2/L^{3}\oplus \cdots \oplus L^{i-1}/L^i$$
which implies
\begin{equation} \label{eq1}
\dim (L/L^{i})=\sum_{k=1}^{i-1}\dim (L^k/L^{k+1})< \infty.
\end{equation}
This means that the definitions of potentially nilpotency of Lie (respectively, Leibniz) algebras and pro-nilpotency of Lie (respectively, Leibniz) algebras are equivalent.

The same arguments due to the isomorphisms of vector spaces
$$L/L^{[i]} \cong L/L^{[2]}\oplus L^{[2]}/L^{[3]}\oplus \cdots \oplus L^{[i-1]}/L^{[i]}$$
assure that potentially solvable Lie (respectively, Leibniz) algebra is a pro-solvable Lie (respectively, Leibniz) algebra and vice versa, i.e., Definitions \ref{def2} and \ref{def3} the above are equivalent.

Note that the quotients $L/L^i$ of a pro-nilpotent Lie (respectively, Leibniz) algebras are finite-dimensional and nilpotent. In particular, any pro-nilpotent Lie (respectively, Leibniz) algebra is finitely generated.

Due to the inclusions $L^{[i]} \subseteq L^{2^{i-1}}$ one has $\bigcap_{i=1}^{\infty}L^{[i]}\subseteq \bigcap_{i=1} ^{\infty}L^{2^{i-1}} \subseteq \bigcap_{i=1} ^{\infty}L^{i}=0$. Therefore, a residually nilpotent algebra is residually solvable.

%Now we introduce the notion of the residually nilpotent map.
\begin{defn} {A linear map $\rho: L \to L$ is called residually nilpotent, if $\bigcap\limits_{i=1}^{\infty}Im\ \rho^i=0$ holds true.}
\end{defn}
Below we introduce the analogue of the notion of nil-independency that has played a crucial role in the description of finite-dimensional solvable Lie algebras (see \cite{Mubarakzjanov1}).

\begin{defn} Derivations $d_{1},d_{2},\dots,d_{n}$ of Leibniz algebra $L$ over a field $\mathbb{F}$ are said to be {residually nil-independent}, if a map $f=\alpha_{1}d_{1}+\alpha_{2}d_{2}+\ldots+\alpha_{n}d_{n}$ is not residually nilpotent for any scalars $\alpha_{1},\alpha_{2},\dots,\alpha_{n}\in \mathbb{F}$. In other words, $\bigcap\limits_{i=1}^{\infty}Im\ f^i=0$ if and only if $\alpha_{1}=\alpha_{2}=\dots=\alpha_{n}=0.$
\end{defn}

Just to recall that low-order Chevalley cohomology groups of Lie algebras is interpreted as follows
 %(see for instance, \cite{Jac}, \cite{Kac})
$${\rm H}^1(L,L)={\rm Der}(L)/{\rm Inder}(L) \quad \mbox{and} \quad {\rm H}^2(L,L)={\rm Z}^2(L,L)/{\rm B}^2(L,L)$$

where the set ${\rm Z}^2(L,L)$ consists of those elements $\varphi\in {\rm Hom}(\wedge^2L, L)$  such that
\begin{equation}\label{Coc}
Z(x,y,z)=[x,\varphi(y,z)] - [\varphi(x,y), z]
+[\varphi(x,z), y] +\varphi(x,[y,z]) - \varphi([x,y],z)+\varphi([x,z],y)=0,
\end{equation}
while ${\rm B}^2(L,L)$ is the set of those elements $\psi\in {\rm Hom}(\wedge^2L, L)$ there exixts $f\in {\rm Hom}(L,L)$ such that
\begin{equation}\label{eq4}
\psi(x,y)=f([x,y])-[f(x),y] - [x,f(y)].
\end{equation}

An analogue of the above for Leibniz algebras is interpreted similarly. The difference is just to replace
${\rm Z}^2(L,L)$ by ${\rm ZL}^2(L,L)$ consisting of elements $\ \psi\in {\rm Hom}(L\otimes L, L)$
such that
$$Z(x,y,z)=[x,\varphi(y,z)] - [\varphi(x,y), z]
+[\varphi(x,z), y] +\varphi(x,[y,z]) - \varphi([x,y],z)+\varphi([x,z],y)=0,$$
 and ${\rm B}^2(L,L)$ must be replaced by ${\rm BL}^2(L,L)$ consisting of $\ \varphi\in {\rm Hom}(L\otimes L, L)$ satisfying the same condition (\ref{eq4}) above.
$${\rm HL}^1(L,L)={\rm Der}(L)/{\rm Inder}(L) \quad \mbox{and} \quad {\rm HL}^2(L,L)={\rm ZL}^2(L,L)/{\rm BL}^2(L,L).$$

In terms of the cohomology groups the notion of completeness of a Lie (respectively, Leibniz) algebra $L$ means that it is centerless and ${\rm H}^1(L,L)=0$ (respectively, ${\rm HL}^1(L,L)=0$). The rigidity of a Lie (respectively, Leibniz) algebra stands for ${\rm H}^2(L,L)=0$ (respectively, ${\rm HL}^2(L,L)=0$).

It was observed that there is only one, so far, complete and rigid solvable Leibniz algebra with the property that the co-dimension of the nilradical is less than the rank. Such an algebra has been given in \cite{O'ktam} by the following table of multiplications:
\begin{equation}\label{O'}L:\left\{\begin{array}{lll}
[e_1,e_1]=e_3,& [e_i,e_1]=e_{i+1},\quad 2\le i\le n-1,\\[1mm]
[e_1,x]=-e_1,& [e_2,x]=-e_2+e_n,\\[1mm]
[x,e_1]=e_1,& [e_i,x]=-(i-1)e_i,\quad 3\le i\le n.
\end{array}\right.\end{equation}
%For the convenience we introduce the notations
%\begin{equation}\label{Leib}
%Leib(a,b,c)=[a,[b,c]]-[[a,b],c]+[[a,c],b],
%\end{equation}
%\begin{equation}\label{Coc}Z(a,b,c)=[a,\varphi(b,c)]-[\varphi(a,b),c]+[\varphi(a,c),b]
%+\varphi(a,[b,c])-\varphi([a,b],c)+\varphi([a,c],b).\end{equation}

We make use results of \cite{Qobil2}, where authors considered the algebra $\mathfrak{m}_0$ with the following table of multiplication on a basis $\{e_1,e_2,\dots,e_n\}$
$$\mathfrak{m}_0: \left\{\left[e_i,e_1\right]=e_{i+1}, \ \ i\geq 2\right.$$
and the following results have been obtained.
\begin{prop} \label{prop1} The space of derivations of the algebra $\mathfrak{m}_0$ is the following:
$$\text{Der}(\mathfrak{m}_0):\left\{\begin{array}{ll}d(e_{1})=\sum\limits_{i=1}^{t}\alpha_{i}e_{i},\\[2mm]
  d(e_{k})=((k-2)\alpha_{1}+\beta_{2})e_{k}+\sum\limits_{i=3}^{t}\beta_{i}e_{i+k-2},\,\,\ \text{where} \,\,k\geq 2.
  \end{array}\right.$$
 \end{prop}

 Note that  $t\in \mathbb{N}$ in Proposition \ref{prop1} stands for the maximal index with non-zero coefficient in the expansion by the basis, i.e., $\alpha_{i}=0$ for all $i > t.$ The same is applied throughout the paper when the upper limit of a summation is given as a variable.

%%{\color{red}{ \begin{rem} In Proposition \ref{prop1}, $t\in \mathbb{N}$ is the number of maximal basis elements whose derivations $d(a) \quad a\in \mathfrak{m}_0/\mathfrak{m}_0^2$ are represented by linear combinations of these basis elements. And throughout the paper the number $t\in \mathbb{N}$ in the multiplication tables of some algebras means as the number of maximal basis elements whose derivations $d(a)$, where $a$ runs the generators of maximal by inclusion pro-nilpotent ideal of those algebras, are represented by them.
%%%  And throughout the paper the number $t\in \mathbb{N}$ in the multiplication tables of algebras comes with such kind of property also.
%\end{rem}}} {\color{blue}{plohoe obyasnenie}}

\begin{lem} Let $M$ be a residually solvable Lie algebra whose maximal by inclusion  pro-nilpotent ideal is $\mathfrak{m}_0$ and $Q$ be the subspace complementary to $\mathfrak{m}_0$. Then one has $\bigcap\limits_{i=1}^{\infty}Im\empty\ ad_{x}^{i}\neq0$ for every $x\in Q$.
\end{lem}

 \begin{thm} Let $M$ be an residually solvable Lie algebra whose maximal by inclusion pro-nilpotent ideal is $\mathfrak{m}_0$ and $Q$ be the subspace complementary to $\mathfrak{m}_0$. Then the dimension of $Q$ is not greater than the maximal number of residually  nil-independent derivations of $\mathfrak{m}_0$.
\end{thm}

\begin{cor} \label{cor} The maximal number of residually nil-independent derivations of $\mathfrak{m}_0$ equals two.
\end{cor}

Therefore, due to Corollary \ref{cor} the dimension of $Q$ is not greater than $2$. The case $\text{dim}Q=2$ has been treated in \cite{Qobil2}.

\section{Main Results}\label{main}
\subsection{Infinite-dimensional Lie algebras}\label{sub1}

%Let Leibniz identity be denoted by $$Leib(x,y,z)=[x,[y,z]]-[[x,y],z] + [[x,z],y].$$

 The section we begin with treating the case $\dim Q=1$ of the discussion at the end of the previous section, i.e., $M=\mathfrak{m}_0\oplus Q$, where $\dim Q=1.$ Further, throughout the paper a residually solvable algebra whose maximal by inclusion pro-nilpotent ideal is $N$ and the codimension of  $N$ is  $n$ is denoted by  $R(N,n)$.

\begin{thm}\label{thm1} Let $R(\mathfrak{m}_0,1)$ be a {residually} solvable Lie algebra whose maximal by inclusion pro-nilpotent ideal is $\mathfrak{m}_0$. Then it admits a basis $\{x, e_{1}, e_{2}, \dots\}$ such that the table of multiplications of $R(\mathfrak{m}_0,1)$ on this basis is given by one of the following form

\qquad \qquad \qquad \qquad \qquad $R_1(\mathfrak{m}_0,1,\beta)$ $:\left\{\begin{array}{lll}
 [e_{i},e_1]=e_{i+1},& i\geq2,\\[1mm]
 [e_1,x]=e_1,\\[1mm]
 [e_i,x]=(i-2)e_i+\sum\limits_{k=3}^{t}\beta_{k}e_{i+k-2},&  i\geq2,
 \end{array}\right.$\\
where $\beta=(\beta_3,\beta_4,\dots,\beta_t)\in\mathbb{C}^{t-2} \ \text{for \ some}\  t\in \mathbb{N}.$\\

\qquad \qquad \qquad \qquad \qquad  $R_2(\mathfrak{m}_0,1,\beta)$ $:\left\{\begin{array}{lll}
[e_{i},e_1]=e_{i+1},& i\geq2,\\[1mm]
[e_1,x]=e_1+\alpha_2e_2,\\[1mm]
[e_i,x]=((i-2)+\beta_2)e_i+\sum\limits_{k=3}^{t}\beta_{k}e_{i+k-2},& i\geq2,
\end{array}\right.$\\
 where $\beta=(\beta_2,\beta_3,\dots,\beta_t)\in\mathbb{C}^{t-1} \ \text{for \ some}\  t\in \mathbb{N}$.\\

\qquad \qquad \qquad \qquad \qquad $R_3(\mathfrak{m}_0,1,\beta)$ $:\left\{\begin{array}{lll}
[e_{i},e_1]=e_{i+1},& i\geq2,\\[1mm]
[e_{i},x]=e_{i}+\sum\limits_{k=3}^{t}\beta_{k}e_{i+k-2},& i\geq 2,
\end{array}\right.$\\
 where $\beta=(\beta_3,\beta_4,\dots,\beta_t)\in\mathbb{C}^{t-2} \ \text{for \ some}\  t\in \mathbb{N}.$
\end{thm}
\begin{proof}
Remind that $R(\mathfrak{m}_0,1)=\mathfrak{m_0}\oplus Q$ is the solvable Lie algebra, where $\mathfrak{m}_0$ is the nilradical of $R(\mathfrak{m}_0,1)$ and $Q$ is the complementary subspace to $\mathfrak{m}_0$. Note that for any $x\in Q$ the inner derivation $ad_x\big|_{\mathfrak{m}_0}$ is non-nilpotent, therefore we can write the table of multiplications of $R(\mathfrak{m}_0,1)$ as follows
$$[e_{1},x]=\sum\limits_{k=1}^{t}\alpha_{k}e_{k}, \quad [e_{i},x]=((i-2)\alpha_{1}+\beta_{2})e_{i}+\sum\limits_{k=3}^{t}\beta_{k}e_{i+k-2},\ \ \text{where} \ k\geq 2.$$

Since $\mathfrak{m}_0$ is maximal by inclusion pro-nilpotent ideal, the algebras in Theorem \ref{thm1} are obtained via case by case consideration with respect to the parameters $\alpha_1$ and $\beta_2$ as follows.

\emph{\bf{Case 1.}} Let $\alpha_1\neq0.$ Then scaling $x$ by $\frac{1}{\alpha_1}$ we get
$$[e_1,x]=e_1+\sum\limits_{k=2}^{t}\alpha_{k}e_{k}, \quad [e_{i},x]=((i-2)+\beta_{2})e_{i}+\sum\limits_{k=3}^{t}\beta_{k}e_{i+k-2},\ \ \text{where} \ k\geq 2.$$

\emph{\bf{Case 1.1.}}
Let $\beta_2=0.$ Setting $e_1'=e_1+\alpha_2e_2$, without loss of generality, we can write
$$[e_1,x]=e_1+\sum\limits_{k=3}^{t}\alpha_{k}e_{k}, \quad
[e_i,x]=(i-2)e_i+\sum\limits_{k=3}^{t}\beta_{k}e_{i+k-2},\  i\geq2.$$
Then the base change $x'=x+\sum\limits_{k=2}^{t}\alpha_{k+1}e_{k}$ (the other basis vectors being unchanged) leads us to the required table of multiplications of $R_1(\mathfrak{m}_0,1,\beta)$. Note that the same is applied through the paper if only one basis vector's change is given.
% $$M_1(\beta):\left\{\begin{array}{lll}
% [e_1,x]=e_1,\\[1mm]
% [e_i,x]=(i-2)e_i+\sum\limits_{k=3}^{t}\beta_{k}e_{i+k-2},\  i\geq2.
% \end{array}\right.$$

\emph{\bf{Case 1.2.}}
Let $\beta_2\neq0.$
Then by changing $x$ as  $x'=x+\sum\limits_{k=2}^{t-1}\alpha_{k+1}e_{k}$,  we get the required table of multiplications of $R_2(\mathfrak{m}_0,1,\beta)$.

%$[e_{1},x']=e_1+\alpha_2e_2,\\[2mm]$
%Thus we have
%$$M_2(\beta):\left\{\begin{array}{lll}
%[e_1,x]=e_1+\alpha_2e_2,\\[1mm]
%[e_i,x]=((i-2)+\beta_2)e_i+\sum\limits_{k=3}^{t}\beta_{k}e_{i+k-2},\ i\geq2.\\[1mm]
%\end{array}\right.$$

\emph{\bf{Case 2.}} Let now $\alpha_1=0$. Then $\beta_2\neq0$ (the otherwise case contradicts the maximal pro-nilpotency of $\mathfrak{m}_0$).
%But we can see the extension algebra is not pro-solvable Lie algebra.  It may be just residually solvable Lie algebra.
If we scale $x$ by $\frac{1}{\beta_2}$ then the relations are easily computed to be
$$[e_{1},x]=\sum\limits_{k=2}^{t}\alpha_{k}e_{k}, \quad [e_{i},x]=e_{i}+\sum\limits_{k=3}^{t}\beta_{k}e_{i+k-2},\  \ i\geq 2.$$

Let set $e_1'=e_1-\alpha_2e_2$. Then without loss of generality we write  $[e_{1},x]=\sum\limits_{k=3}^{t+1}\alpha_{k}e_{k}$
and applying the base change $x'=x+\sum\limits_{k=2}^{t}\alpha_{k+1}e_{k}$ we get the table of multiplication as in $R_3(\mathfrak{m}_0,1,\beta)$.
%Thus we have
%$$M_3(\beta):\left\{\begin{array}{lll}
%[e_{i},e_1]=e_{i+1},& i\geq2,\\[1mm]
%[e_{i},x]=e_{i}+\sum\limits_{k=3}^{t}\beta_{k}e_{i+k-2},& i\geq 2.
%\end{array}\right.$$
\end{proof}
The following proposition shows that the algebras {{$R_1(\mathfrak{m}_0,1,\beta)$, $R_2(\mathfrak{m}_0,1,\beta)$ and $R_3(\mathfrak{m}_0,1,\beta)$}} are not complete.

\begin{prop}{\label{lem1}} \emph{}

%The algebras
%$R_1(\mathfrak{m}_0,1)$, $R_2(\mathfrak{m}_0,1)$ and $R_3(\mathfrak{m}_0,1)$ have outer derivations given by
\begin{itemize}
  \item $d(e_i)=e_i,\ i\geq2$ is an outer derivation of the algebra {{$R_1(\mathfrak{m}_0,1,\beta)$}};
\item The derivations
      \begin{itemize}
  \item[$\ast$] $\left\{\begin{array}{ll}
d(e_1)=e_2,\\[1mm]
d(x)=\sum\limits_{i=2}^{s-1}\beta_{i+1}e_i,\\[1mm]
\end{array}\right.$ \quad if $\beta_2=1$,\\
 and
 \item[$\ast$] $\left\{\begin{array}{ll}
d(e_1)=\frac{\alpha_2}{\beta_2-1}e_2,\\[1mm]
d(e_i)=e_i,\ i\geq2,\\[1mm]
d(x)=\sum\limits_{i=2}^{s-1}\frac{\alpha_2\beta_{i+1}}{\beta_2-1}e_i,\\[1mm]
\end{array}\right.$ if $\beta_2\neq 1$\\
are outer derivations of the algebra $R_2(\mathfrak{m}_0,1,\beta)$, respectively
\end{itemize}
%\begin{itemize}
  \item $d(e_i)=e_{i+2},\ i\geq2,$ is an outer derivation of the algebra $R_3(\mathfrak{m}_0,1,\beta)$.
%\end{itemize}
\end{itemize}
\end{prop}
\begin{proof} It is obvious that $d(e_i)=e_i,\ i\geq2$ is a derivation of $R_1(\mathfrak{m}_0,1,\beta)$. Moreover, for any {{$a \in R_1(\mathfrak{m}_0,1,\beta)$ we have $ad_a(R_1(\mathfrak{m}_0,1,\beta))\subseteq[R_1(\mathfrak{m}_0,1,\beta),R_1(\mathfrak{m}_0,1,\beta)]=R_1(\mathfrak{m}_0,1,\beta)\setminus\{e_2\}$. Thus, $d$ is not inner derivation of $R_1(\mathfrak{m}_0,1,\beta)$.}}
%So $d$ is an outer derivation of the algebra $M_1(\beta)$.

The similar argument can be applied to prove the rest parts of the proposition.
\end{proof}
%\begin{proof} By checking derivation properties, we can check that $d$ defined by $d(e_i)=e_{i+2},\ i\geq2,$ is a derivation of the algebra $M(\beta)_3$. Assume the contrary. Suppose, $d$ is an inner derivation. Then there exists $c=\sum\limits_{i=1}^{s}A_ie_i+Ax$ such that $ad_c=d$.
%$ad_c(e_1)=\sum\limits_{i=2}^{s}A_ie_{i+1}=0,\ \Rightarrow \ A_i=0,\ 2\le i\le s,\ \Rightarrow \ c=A_1e_1+Ax$ and
%$ad_c(e_2)=-A_1e_3-A(e_2+\sum\limits_{k=3}^{t}\beta_ke_k)\neq e_4.$ It is a contradiction.
%\end{proof}

Recall that in \cite{Qobil2},
 %{\color{red} {(the tables of multiplications of a family of Lie algebras denoted by $R(\beta)$ has been given. Note that $R(\beta)$ is a residually solvable Lie algebra with property that the subspace complementary to $\mathfrak{m}_0$ is maximal. The theorem below is on derivations of $R(\beta)$.)}} {\color{green}{
 the following residually solvable Lie algebra $R(\mathfrak{m}_0,2,\beta)$, with a property that the codimension of $\mathfrak{m}_0$ is maximal, was obtained
$$R(\mathfrak{m}_0,2,\beta):\left\{\begin{array}{ll}
[e_{i} ,e_{1}]=e_{i+1},\quad i\geq2,\\[1mm]
[e_1,x]=e_1,\\[1mm]
[e_i,x]=(i-1)e_i+\sum\limits_{q=3}^{t}\beta_{q}e_{q+i-2},\quad i\geq 2,\\[1mm]
[e_i,y]=e_i,\quad i\geq 2,
 \end{array}\right.$$
where $\beta=(\beta_3,\beta_4,\dots,\beta_t)\in\mathbb{C}^{t-2} \ \text{for \ some}\  t\in \mathbb{N},$ where $\{x,y\}$ is the basis of $Q$ and $\{e_1, e_2,\dots\}$ is basis of $\mathfrak{m}_0$. Here is a theorem on the derivations of $R(\mathfrak{m}_0,2,\beta).$

\begin{thm}\label{thm3_3}\label{t1} An arbitrary algebra of the family $R(\mathfrak{m}_0,2, \beta)$ is complete.
\end{thm}
\begin{proof}  The fact that the center is trivial can be easily obtained by using the table multiplications. We prove that all derivations of $R(\mathfrak{m}_0,2, \beta)$ are inner.
Note that $\{e_1, e_2, x,y\}$ are generators of $R(\mathfrak{m}_0,2,\beta)$. Since a derivation is completely determined by its values on generators it is sufficient to prove the existence of $a\in R(\mathfrak{m}_0,2,\beta)$ such that $d(z)=ad_a(z)$, where $z$ is either of $e_1, e_2, x, y$.

Observe that for any $k\in \mathbb{N}$ the quotient algebra
%({\color{red}{$$R_k(\beta):=R(\beta)/\mathfrak{m}_0^{k}=\mathfrak{m}_0/\mathfrak{m}_0^k\oplus Q=\overline{\mathfrak{m}}_0\oplus Q \ \mbox{with} \ \mathfrak{m}_0^k=<e_{k+1}, \dots >$$}})
$$(R(\mathfrak{m}_0,2,\beta))_k:=R(\mathfrak{m}_0,2,\beta)/\mathfrak{m}_0^{k}=\mathfrak{m}_0/\mathfrak{m}_0^k\oplus Q=\overline{\mathfrak{m}}_0\oplus Q \ \mbox{with} \ \mathfrak{m}_0^k=\mathrm{Span}\{e_{k+1}, \dots \}$$
is finite-dimensional solvable Lie algebra, which is maximal solvable extension of the nilpotent Lie algebra $\mathfrak{m}_0$. Since such an algebra is unique up to isomorphism the algebra $(R(\mathfrak{m}_0,2,\beta))_k$ must be isomorphic to $(R(\mathfrak{m}_0,2,0))_k$. However, according to \cite{Ancochea1}, all its derivations are inner.

Let $d\in Der(R(\mathfrak{m}_0,2, \beta))$ and introduce $\bar{d} \in Der((R(\mathfrak{m}_0,2,\beta))_k)$ such that $\bar{d}(\bar{v})=\overline{d(v)},\ \ \bar{v}=v+\mathfrak{m}_0^k.$ The function $\bar{d}(\bar{v})$ is well-defined, i.e., $\mathfrak{m}_0^k$ is invariant under $d$.
% ($d(\mathfrak{m}_0^k)\subseteq \mathfrak{m}_0^k$).

Indeed,
$$d(\mathfrak{m}_0)=d_{\mathfrak{m}_0}(\mathfrak{m}_0)+d_Q(\mathfrak{m}_0),\ \mbox{where} \  d_{\mathfrak{m}_0} : \mathfrak{m}_0 \rightarrow \mathfrak{m}_0 , \quad d_Q : \mathfrak{m}_0 \rightarrow Q$$

Taking into account that $[Q,Q]=0$ and $\mathfrak{m}_0$ is an ideal of $R(\mathfrak{m}_0,2, \beta)$ such that $[\mathfrak{m}_0,Q]=\mathfrak{m}_0$ we derive
$$d(\mathfrak{m}_0)=d([\mathfrak{m}_0,Q])=[d(\mathfrak{m}_0),Q]+[\mathfrak{m}_0,d(Q)]=
[d_{\mathfrak{m}_0}(\mathfrak{m}_0),Q]+[\mathfrak{m}_0,d(Q)]\subseteq \mathfrak{m}_0.$$

Consequently, we get $$d_Q(\mathfrak{m}_0)=0, \quad d(\mathfrak{m}_0)\subseteq \mathfrak{m}_0.$$
These imply $d(\mathfrak{m}_0^k)\subseteq \mathfrak{m}_0^k$ for any $k\in \mathbb{N}.$ Thus the well-definedness of $\bar{d}$ is shown.
%which prove the correctness of definition of induced derivation $\bar{d}$.

Set
$$d(e_1)=\sum\limits_{i=1}^{s}a_{i}e_i, \quad d(e_2)=\sum\limits_{i=1}^{s}b_{i}e_i, \quad
  d(x)=\sum\limits_{i=1}^{s}\gamma_{i}e_i+\gamma_{1,1}x+\gamma_{2,2}y, \quad d(y)=\sum\limits_{i=1}^{s}\tau_{i}e_i+\tau_{1,1}x+\tau_{2,2}y.
$$
%??????????
%$d([e_1,y])=[d(e_1),y]+[e_1,d(y)]=\sum\limits_{i=2}^{s}a_{i}e_i-\sum\limits_{i=2}^{s}\tau_{i}e_{i+1}=0$ implies that $a_2=\tau_s=0$ and $\tau_i=a_{i+1},\quad 2\le i\le s-1$.
%????????????????????????
Let us take $k\geq \max\{s,t\}$. Then $\overline{d(v)}=\bar{d}(\bar{v})=ad_{\bar{c}_k}$ for some $\bar{c}_k=c_k+\mathfrak{m}_0^k$ and for any $\bar{v}=v+\mathfrak{m}_0^k.$

Let
$$\bar{c}_k=\sum\limits_{i=1}^{k}\alpha_{k,i}e_i+\lambda_{k}x+\mu_{k}y+\mathfrak{m}_0^k.$$
Then from the equalities
$$\overline{d(e_1)}=[\overline{e}_1,\bar{c}_k], \quad \overline{d(e_2)}=[\overline{e}_2,\bar{c}_k], \quad \overline{d(y)}=[\overline{y},\bar{c}_k]$$
we get
$$\begin{array}{lll}
\sum\limits_{i=1}^{s}a_{i}e_i-\lambda_ke_1+\sum\limits_{i=2}^{k-1}\alpha_{k,i}e_{i+1}\in \mathfrak{m}_0^k,\\[1mm]
\sum\limits_{i=1}^{s}b_{i}e_i-(\lambda_k+\mu_k)e_2-(a_{k,1}+\beta_3)e_3-
\lambda_k\sum_{q=4}^{t}\beta_qe_q\in\mathfrak{m}_0^k,\\[1mm]
%$$\sum\limits_{i=1}^{s}\gamma_{i}e_i+\gamma_{1,1}x+\gamma_{2,2}y+\mathfrak{m}_0^k=\overline{d(x)}=[\overline{x},\bar{a}_k]=[x,a_k]+\mathfrak{m}_0^k=$$
%$$=-\alpha_{k,1}e_1+\sum\limits_{i=2}^{k}\alpha_{k,i}((1-i)e_i-\sum\limits_{q=3}^{t}\beta_{q}e_{q+i-2})+\mathfrak{m}_0^k=$$
%$$=-\alpha_{k,1}e_1+\sum\limits_{i=2}^{k}\alpha_{k,i}((1-i)e_i-\sum\limits_{q=3}^{k-i+2}\beta_{q}e_{q+i-2})+\mathfrak{m}_0^k,$$
\sum\limits_{i=1}^{s}\tau_{i}e_i+\tau_{1,1}x+\tau_{2,2}y+\sum\limits_{i=2}^{k}\alpha_{k,i}e_i\in \mathfrak{m}_0^k.\end{array}$$
%$$\tau_1e_1+\sum\limits_{i=2}^{s-1}a_{i+1}e_i+\tau_{1,1}x+\tau_{2,2}y+\sum\limits_{i=2}^{k}\alpha_{k,i}e_i\in \mathfrak{m}_0^k.$$
Comparing the coefficients of the basis vectors we derive
\begin{equation}\label{eq2}
\left\{\begin{array}{llll}
a_1=\lambda_k,\quad  a_i=-\alpha_{k,i-1},\quad 3\leq i\leq s,\quad \alpha_{k,i}=0,\quad s\leq j\leq k-1, \\[1mm]
b_2=\lambda_k+\mu_k,\quad  b_3=\alpha_{k,1}+\beta_3. \\[1mm]
\tau_i=-\alpha_{k,i},\quad 2\leq i\leq s-1,\quad \alpha_{k,k}=0,\\[1mm]
\end{array}\right.
\end{equation}

From \eqref{eq2} we obtain $c_k=(b_3-\beta_3)e_1-\sum\limits_{i=2}^{s-1}a_{i+1}e_i+a_1x+(b_2-a_1)y$. Hence, $c_k$ depends only on the parameters $a_i,\tau_i, \ b_1, \ b_2,\ 1\le i\le s.$ Therefore, $c_k=c_{k+1}$ for any $k\geq \max\{s,t\}.$ Thus, setting $c:=c_k$ and $W_k=\mathrm{Span}\{x, y,  e_1, \dots, e_{k}\}$ we get
%$c_k=(b_3-\beta_3)e_1-\sum\limits_{i=2}^{s-1}a_{i+1}e_i+(b_2+a_1)x-a_1y.$
$d(z)_{|W_k}=ad_{c}(z)_{|W_k} \ \mbox{for} \ z\in \{e_1, e_2, x,y\} \ \mbox{and} \  k\geq \max\{s,t\}.$
Now taking into account that $\bigcup_{k=1}^{\infty}W_k=R(\mathfrak{m}_0,2,\beta)$ we obtain $d=ad_{c}.$
\end{proof}
Let us now treat the second cohomology groups of $R(\mathfrak{m}_0,2,\beta)$.
\begin{thm}\label{thm5} The second cohomology groups %{\color{red}{($H^2(R(\beta),R(\beta))$ of $R(\beta)$)}}
$H^2(R(\mathfrak{m}_0,2,\beta),R(\mathfrak{m}_0,2,\beta))$ of $R(\mathfrak{m}_0,2,\beta)$ are trivial.
\end{thm}
\begin{proof} Let $\varphi\in Z^2(R(\mathfrak{m}_0,2,\beta),R(\mathfrak{m}_0,2,\beta))$. We prove that there exists a $f\in \text{Hom}(R(\mathfrak{m}_0,2,\beta),R(\mathfrak{m}_0,2,\beta))$ that generates $\varphi$ as a coboundry. An element $\varphi$ of $Z^2(R(\mathfrak{m}_0,2,\beta),R(\mathfrak{m}_0,2,\beta))$ on the basis $\{x,y, e_1, e_2, \dots\}$ is written in the form
$$\begin{array}{lll}
\varphi(e_i,e_j)=\sum\limits_{k=1}^{p_{i,j}}a^{i,j}_ke_k+a^{i,j}_{1,1}x+a^{i,j}_{2,2}y, & \varphi(e_i,x)=\sum\limits_{k=1}^{p_{i}}b^{i}_ke_k+b^{i}_{1,1}x+b^{i}_{2,2}y,\\[2mm]
\varphi(e_i,y)=\sum\limits_{k=1}^{q_{i}}c^{i}_ke_k+c^{i}_{1,1}x+c^{i}_{2,2}y,&
\varphi(x,y)=\sum\limits_{k=1}^{s}g_ke_k+g_{1,1}x+g_{2,2}y, \quad i\geq 1.\\[2mm]
\end{array}$$

We choose $f\in \text{Hom}(R(\mathfrak{m}_0,2,\beta),R(\mathfrak{m}_0,2,\beta))$ as follows
 $$\begin{array}{ll}
f(e_1)=\tau_1^1e_1-c_2^1e_2-\sum\limits_{k=3}^{s_1}\tau_k^1e_k+b_{1,1}^1x+b_{2,2}^1y,\\[2mm]
f(e_i)=c_1^ie_1+\sum\limits_{k=2}^{s_i}\tau_k^ie_k+c_{1,1}^ix+c_{2,2}^iy,\\[2mm]
f(x)=\mu_1e_1+\sum\limits_{k=2}^{s_1-1}(b_{k+1}^1-c_2^1\beta_{k+1}+(k-1)\tau_{k+1}^1+\sum\limits_{i=3}^{k}\tau_i^1\beta_{k+3-i})e_k-b_1^1x+\mu_{2,2}y,\\[2mm]
f(y)=g_1e_1+\sum\limits_{k=2}^{s}(c_{k+1}^1+\tau_{k+1}^1)e_k-c_1^1x+\nu_{2,2}y, \ s=max\{q_1,s_1\}.\\
\end{array}$$

Consider the cocycle $\chi=\varphi-\psi\in Z^2(R(\mathfrak{m}_0,2,\beta),R(\mathfrak{m}_0,2,\beta))$, where $\psi(x,y)=f([x,y])-[f(x),y] - [x,f(y)]$ and show that $\chi $ is trivial. In the expension of the cocycle $\chi$ via the basis its components will be written as functions of the components of $f$ and $\varphi$
%{\color{red}{Let expand $\chi $ via the basis
$$\begin{array}{ll}
\chi(e_1,x)=b_2^1e_2,\quad  \chi(e_1,y)=c_{1,1}^{1}x+c_{2,2}^{1}y,\\[1mm]
\chi(x,y)=\sum\limits_{k=2}^{s}g_k'e_k+g_{1,1}'x+g_{2,2}'y,\\[1mm]
\chi(e_i,e_j)=\sum\limits_{k=1}^{p_{i,j}}a^{i,j}_ke_k+a^{i,j}_{1,1}x+a^{i,j}_{2,2}y,\quad i\geq 1,\\[1mm]
\chi(e_i,x)=\sum\limits_{k=1}^{p_{i}}b^{i}_ke_k+b^{i}_{1,1}x+b^{i}_{2,2}y,\quad i\geq 2,\\[1mm]
\chi(e_i,y)=\sum\limits_{k=2}^{q_{i}}c^{i}_ke_k+g_1e_{i+1}-c_1^1((i-2)e_i+\sum\limits_{k=3}^{t}\beta_ke_{k+i-2})+\nu_{2,2}e_i, \quad i\geq2.\\[1mm]
\end{array}$$

Now we impose to $\chi$ the cocycle identities $Z=0$ (see (\ref{Coc})) to derive the following set of constraints:

\begin{center}
    \begin{tabular}{lll}
       \qquad\quad 2-cocyle identity & &\qquad\qquad\qquad Constraints\\
        \hline \hline
\\
        $Z(e_1,x,y)=0,$ &\quad $\Rightarrow $\quad & $\left\{\begin{array}{ll}
                                                                 g_k'=0,\ 2\le k\le s, \ g_{1,1}'=b_2^1=c_{1,1}^1=c_{2,2}^1=0,\\[1mm]

                                                                 \end{array}\right.$\\[1mm]
        $Z(e_i,e_j,y)=0,\ i,j\geq2, $ &\quad $\Rightarrow $\quad & $\left\{\begin{array}{ll}
                                                                 a_k^{i,j}=a_{1,1}^{i,j}=a_{2,2}^{i,j}=0,\  i,j\geq2, \ k\geq1,\\[1mm]
                                                                 \end{array}\right.$\\[1mm]
        $Z(e_i,e_1,y)=0,\ i\geq2, $ &\quad $\Rightarrow $\quad & $\left\{\begin{array}{ll}
                                                                 a_1^{i,1}=a_{1,1}^{i,1}=a_{2,2}^{i,1}=0, \ i\geq2,\\[1mm] c_k^k=c_2^2+(k-2)c_1^1,\ k\geq3,\\[1mm]
                                                                 c_{k}^i=0,\ i\geq3,\ 2\le k\le i-1,\\[1mm]
                                                                  c_k^i=c_{k-i+2}^{2},\ k>i\geq3,\\[1mm]
 \end{array}\right.$\\[1mm]
 $Z(e_i,x,y)=0,\ i\geq2, $ &\quad $\Rightarrow$ \quad &$\left\{\begin{array}{ll}
                                                                  g_{2,2}=b_1^{i}=b_{1,1}^i=b_{2,2}^i=0,\ i\geq2.\\[1mm]
                                                                  \end{array}\right.$

             \end{tabular}
\end{center}

We let
$$\begin{array}{lll}
\tau_{i+1}^{i+1}=\tau_{i}^{i}+\tau_{1}^1+a_{i+1}^{i,1}+b_{1,1}\beta_3,\\[1mm]
\tau_{i}^{i+1}=\tau_{i-1}^{i}+a_{i}^{i,1}+(i-1)b_{1,1}+b_{2,2},\quad
\tau_{2}^{i+1}=a_{2}^{i,1},\\[1mm]
\tau_{k}^{i+1}=\tau_{k-1}^{i}+a_{k}^{i,1}+b_{1,1}\beta_{k+2-i},\quad 3\leq k\leq t+i-2,\ k\neq \{i,i+1\},\quad i\geq 2,
\end{array}$$
to get  $(\varphi-\psi)(e_i,e_1)=\chi(e_i,e_1)=0,\ i\geq 1$.

Letting
$$\begin{array}{ll}
\mu_{2,2}=b_1^1-b_2^2+\sum\limits_{j=3}^{t}\beta_j\tau_2^j,\\[1mm]
\tau_{3}^2=b_1^1\beta_3-\mu_1-b_3^2-\tau_2^2\beta_3+\sum\limits_{j=3}^{t}\beta_j\tau_3^j,\\[1mm]
\tau_{i}^2=\frac{1}{i-2}(b_1^1\beta_i-b_i^2-\tau_2^2\beta_i-\tau_3^2\beta_{i-1}-\dots-\tau_{i-1}^2\beta_3+\sum\limits_{j=3}^{t}\beta_j\tau_i^j),\quad i\geq 4.
\end{array}$$
we obtain $(\varphi-\psi)(e_2,x)=\chi(e_2,x)=0$.

Finally, $Z(e_i,e_1,x)=0,\ i\geq 2$ gives $(\varphi-\psi)(e_i,x)=\chi(e_i,x)=0,\quad i\geq3$.
This completes the proof.
\end{proof}
\subsection{Infinite-dimensional Leibniz algebra.}\label{sub2}
In this section we study the extensions of a non-Lie Leibniz algebra with the following table of multiplications. (This algebra has been introduced in \cite{Omirov}).
%Let consider an infinite dimensional analog of a filiform Leibniz algebra introduced in the paper \cite{Omirov}
$$F:\left\{\begin{array}{ll}
[e_i,e_1]=e_{i+1},&i\geq2.
\end{array}\right.$$
The algebra of derivations of $F$ is given by the following proposition, the proof of that is immediate if we apply the derivation rules to the table of multiplications above.
\begin{prop}\label{prop01} The derivations of the algerba $F$ are given as follows:
$$Der(F): \quad \left\{\begin{array}{lll}
d(e_1)=\alpha_1e_1,\\[1mm]
d(e_i)=((i-2)\alpha_1+\beta_2)e_{i}+\sum\limits_{k=3}^t\beta_{k}e_{k+i-2},\ \ i\geq2,\\[1mm]
\end{array}\right.$$
\end{prop}
%\begin{proof} This proposition is proved by using multiplication table of the algebra $F$ and by checking the derivation properties.
%$$d(e_1)=\sum\limits_{k=1}^t\alpha_{k}e_{k},\quad d(e_2)=\sum\limits_{k=1}^t\beta_{k}e_{k}.$$
%
%$d([e_1,e_1])=[d(e_1),e_1]+[e_1,d(e_1)]=\sum\limits_{k=2}^t\alpha_ke_{k+1}=0,\  \alpha_k=0,\  2\le k\le t,$
%
%$d([e_2,e_2])=[d(e_2),e_2]+[e_2,d(e_2)]=\beta_1e_3=0,\ \beta_1=0,$
%
%$d(e_3)=d([e_2,e_1])=[d(e_2),e_1]+[e_2,d(e_1)]=(\alpha_1+\beta_2)e_{3}+\sum\limits_{k=3}^t\beta_{k}e_{k+1},$
%
%Then we can show
% $d(e_i)=((i-2)\alpha_1+\beta_2)e_{i}+\sum\limits_{k=3}^t\beta_{k}e_{k+i-2},\ i\geq2$,
%
%by using induction, i.e.,
%
%$d(e_{i+1})=d([e_i,e_1])=[d(e_i),e_1]+[e_i,d(e_1)]=((i-2)\alpha_1+\beta_2)e_{i+1}+\sum\limits_{k=3}^t\beta_{k}e_{k+i-1}+\alpha_1e_{i+1}=
%((i-1)\alpha_1+\beta_2)e_{i}+\sum\limits_{k=3}^t\beta_{k}e_{k+i-1}.$
%
%\end{proof}

Let $R(F)$ stand for a family of residually solvable Leibniz algebras whose maximal by inclusion  pro-nilpotent ideal is $F$ and $Q$ be the subspace complementary to $F$, i.e., $R(F)=F\oplus Q$.

\begin{lem}\label{lem01} The derivations $ad_{x}$ are non-residually nilpotent for any $x\in Q$.
\end{lem}
\begin{proof} Let us assume the contrary that $\bigcap\limits_{k=1}^{\infty}Im\empty\ ad_{x}^{k}=0$ for some $x\in Q$. Set $R(F)=F+\mathbb{C}x$. Since $ad_{{x}|F}=d$ for some $d\in Der(F)$ the condition $\bigcap\limits_{k=1}^{\infty}Im\empty\ ad_{x}^{k}=0$ implies $\alpha_{1}=\beta_{2}=0$ (see Proposition \ref{prop01}). Therefore, $[e_{1},x]=0, \quad [e_{i},x]=\sum\limits_{k=3}^{t}\beta_{k}e_{i+k-2}, \ \ i\geq 2$
and $\bigcap\limits_{i=1}^{\infty}(R(F))^{i}=0$, i.e., $R(F)$ is pro-nilpotent which contradicts to the maximality of $F.$
\end{proof}

%\begin{prop}\label{thm01} The dimension of $Q$ is not greater than the maximal number of residually nil-independent derivations of $F$.
%\end{prop}

The fact that the maximal number of residually nil-independent derivations of $F$ is equal to $2$ implies that the dimension of the subspace $Q$ of $R(F)$ is not greater than $2$.
%the maximal number of residually nil-independent derivations of $F$. . Therefore, we can conclude, the dimension of $Q$ is no greater than
%{\color{red}{ quyidagicha belgilashni Introduction part ni oxiriga o'tkazib qo'ydim
%\\
%
%We denote by $R(F,i)$ the family of residually solvable Leibniz algebras whose maximal by inclusion pro-nilpotent ideal is $F$ and $R(F,i)=F\oplus Q$, where $\text{dim}Q=i$ and $i=1,2.$}}
Let consider the case $\dim Q=1.$
\begin{thm}\label{thm02}  The algebra $R(F,1)$ admits a basis $\{x, e_{1}, e_{2}, \dots\}$ such that the table of multiplications of $R(F,1)$ on this basis has one of the following forms:

$\qquad \qquad \qquad \qquad \qquad R_1(F,1,\beta):\left\{\begin{array}{ll}
[e_{i} ,e_{1}]=e_{i+1},\quad i\geq2,\\[1mm]
[e_1,x]=-[x,e_1]=e_1,\\[1mm]
[e_i,x]=(i-2+\beta_2)e_{i}+\sum\limits_{k=3}^t\beta_{k}e_{k+i-2},\ \ i\geq2.
\end{array}\right.$\\
where $\beta=(\beta_2,\beta_3,\dots,\beta_t)\in\mathbb{C}^{t-1} \ \text{for \ some}\  t\in \mathbb{N}$.

$\qquad \qquad \qquad \qquad \qquad R_2(F,1,\beta):\left\{\begin{array}{lll}
[e_{i} ,e_{1}]=e_{i+1},\quad i\geq2,\\[1mm]
[x,e_1]=-e_1+e_2,\\[1mm]
[e_1,x]=e_1,\\[1mm]
[e_i,x]=(i-1)e_{i}+\sum\limits_{k=3}^t\beta_{k}e_{k+i-2},\ \ i\geq2,\\[1mm]
%[x,e_i]=0,\ \ i\geq2,\\[1mm]
[x,x]=\sum\limits_{k=2}^{t-1}\beta_{k+1}e_k.\\[1mm]
\end{array}\right.$\\
where $\beta=(\beta_3,\beta_4,\dots,\beta_t)\in\mathbb{C}^{t-2} \ \text{for \ some}\  t\in \mathbb{N}$.

$\qquad \qquad \qquad \qquad \qquad R_3(F,1,\beta):\left\{\begin{array}{lll}
[e_{i} ,e_{1}]=e_{i+1},\quad i\geq2,\\[1mm]
[e_i,x]=e_i+\sum\limits_{k=3}^t\beta_{k}e_{k+i-2},\ \ i\geq2,\\[1mm]
\end{array}\right.$\\
where $\beta=(\beta_3,\beta_4,\dots,\beta_t)\in\mathbb{C}^{t-2} \ \text{for \ some}\  t\in \mathbb{N}$.
\end{thm}
\begin{proof} Since, for any $x\in Q$ the derivation $ad_x$ of $F$ is non-nilpotent we use it to write the products $[e_1,x]$ and $[e_i,x]\ \ i \geq 2$ in $F$ as follows
%, i.e.,}} (By using {\color{red}{the derivation rule}} we obtain)
$$\begin{array}{ll}
[e_1,x]=\alpha_1e_1,& [e_i,x]=((i-2)\alpha_1+\beta_2)e_{i}+\sum\limits_{k=3}^t\beta_{k}e_{k+i-2},\ \ i\geq 2.
\end{array}$$
Let
$$\begin{array}{ll}
[x,e_i]=\sum\limits_{k=1}^{t}\gamma_{i,k}e_k,\ \ i\in\{1,2\}, & [x,x]=\sum\limits_{k=1}^{t}\mu_{k}e_k+\mu x.\\[1mm]
\end{array}$$
%
%Since $F$ is maximal by inclusion pro-nilpotent ideal, the algebras in Theorem \ref{thm02} can be obtained by case by case considerations with respect to the parameters $\alpha_1$ and $\beta_2$ as follows. %(\emph{case 1.} $\alpha_1\neq0,$ \emph{case 2.} $\alpha_1=0$, $\beta_2\neq0$)
 %similar to the proof of  {Theorem \ref{thm1}}.

\emph{\bf{Case 1.}} Let $\alpha_1\neq0.$
Taking $x'=\frac{x}{\alpha_1}$ we write $[e_1,x]=e_1.$
Since $[e_i,[x,x]]=0,\ i\in\{1,2\}$, we get $\mu=\mu_1=0.$
And $Leib(x,e_1,e_2)=Leib(e_2,x,e_2)=0$ gives $[x,e_2]=0.$

Note that the base change $x'=x-\sum\limits_{k=2}^{t-1}\gamma_{1,k+1}e_k$ gives $[x',e_1]=\gamma_{1,1}e_1+\gamma_{1,2}e_2.$ Since $e_1\not \in Ann_r(R(F,1))$ and $e_i \in Ann_r(R(F,1)),$ $i \geq 3$ we obtain $\gamma_{1,1}=-1$ and $[x,e_i]=0$ for $i\geq 3.$

By using $Leib(x,x,e_1)=0$, we get
$$\begin{array}{lll}
\gamma_{1,2}(\beta_2-1)=0,\quad \mu_{k}=\gamma_{1,2}\beta_{k+1},\quad 2\le k\le t-1,\quad
\mu_{t}=0.
\end{array}$$

\emph{\bf{Case 1.1.}} If  $\gamma_{1,2}=0$ then $\mu_{k}=0,\ 2\le k\le t$ and this leads to $R_1(F,1,\beta)$. %consequently.
%$$\left\{\begin{array}{llll}
%[e_1,x]=-[x,e_1]=e_1,\\[1mm]
%[e_i,x]=(i-2+\beta_2)e_{i}+\sum\limits_{k=3}^t\beta_{k}e_{k+i-2},\ \ i\geq2.
%\end{array}\right.$$

\emph{\bf{Case 1.2.}} If $\gamma_{1,2}\neq0$ then $\beta_2=1$. The base change
$e_1'=e_1$ and $e_i'=\gamma_{1,2}e_i,\ \ i\geq2$ leads to the table of multiplications of $R_2(F,1,\beta)$.
%$$\left\{\begin{array}{lll}
%[x,e_1]=-e_1+\gamma_{1,2}e_2,\\[1mm]
%[e_1,x]=e_1,\\[1mm]
%[e_i,x]=(i-1)e_{i}+\sum\limits_{k=3}^t\beta_{k}e_{k+i-2},\ \ i\geq2,\\[1mm]
%%[x,e_i]=0,\ \ i\geq2,\\[1mm]
%[x,x]=\gamma_{1,2}\sum\limits_{k=2}^{t-1}\beta_{k+1}e_k.\\[1mm]
%\end{array}\right.$$
%
%Let $e_1'=e_1,\ e_i'=\gamma_{1,2}e_i,\ \ i\geq2.$ Then we derive
%$$\left\{\begin{array}{lll}
%[x,e_1]=-e_1+e_2,\\[1mm]
%[e_1,x]=e_1,\\[1mm]
%[e_i,x]=(i-1)e_{i}+\sum\limits_{k=3}^t\beta_{k}e_{k+i-2},\ \ i\geq2,\\[1mm]
%%[x,e_i]=0,\ \ i\geq2,\\[1mm]
%[x,x]=\sum\limits_{k=2}^{t-1}\beta_{k+1}e_k.\\[1mm]
%\end{array}\right.$$

\emph{\bf{Case 2.}} Let now $\alpha_1=0.$ Then $\beta_2\neq 0$ (otherwise, $F$ would be non-maximal pro-nilpotent ideal %{\color{green}{ of a  residually solvable Lie algebra $R$}}
). Let $x'=\frac{x}{\beta_2}$ to get
$$\begin{array}{lll}
[e_i,x]=e_i+\sum\limits_{k=3}^t\beta_{k}e_{k+i-2},\quad i\geq2,&
[x,e_i]=\sum\limits_{k=1}^{t}\gamma_{i,k}e_k,\quad  i\in\{1,2\},&
[x,x]=\sum\limits_{k=1}^{t}\mu_{k}e_k+\mu x.
\end{array}$$
Since $[e_2,[x,x]]=0$
%\mu_1e_3+\mu e_2+\mu\sum\limits_{k=3}^t\beta_{k}e_{k}=0,\ \Rightarrow \
we obtain $\mu=\mu_1=0$, hence, $[x,x]=\sum\limits_{k=2}^{t}\mu_{k}e_k.$

The base change $x'=x-\sum\limits_{k=2}^{t-1}\gamma_{1,k+1}e_k,$ gives $[x',e_1]=\gamma_{1,1}e_1+\gamma_{1,2}e_2.$
Moreover, $Leib(e_2,x,e_1)=Leib(x,e_1,e_2)=0$ implies
%[e_2,[x,e_1]]-[[e_2,x],e_1]+[[e_2,e_1],x]=0,\Rightarrow\ \gamma_{1,1}e_3-e_3-\sum\limits_{k=3}^t\beta_{k}e_{k+1}+e_3+\sum\limits_{k=3}^t\beta_{k}e_{k+1}=0,\ \Rightarrow \
$\gamma_{1,1}=\gamma_{2,k}=0,\  2\le k\le t.$ Hence, $[x,e_2]=\gamma_{2,1}e_1.$

%$$[x,[e_1,e_2]]-[[x,e_1],e_2]+[[x,e_2],e_1]=0,\Rightarrow 0-0+\sum\limits_{k=2}^{t}\gamma_{2,k}e_{k+1}=0,\ \Rightarrow \ ;$$
%$$[x,e_2]=\gamma_{2,1}e_1;$$
%$$[x,e_3]=[x,[e_2,e_1]]=[[x,e_2],e_1]-[[x,e_1],e_2]=0;$$
Since  $[a,b]+[b,a]\in Ann_r(R(F,1))$ for any $a,b \in R(F,1)$, we get $e_i\in R(F,1),\ i\geq3.$ This means that $[x,e_i]=0,\ i\geq3.$
%$[x,e_3]=[x,[e_2,e_1]]=[[x,e_2],e_1]-[[x,e_1],e_2]=0$. Let suppose $[x,e_i]=0$, then $[x,e_{i+1}]=[x,[e_i,e_1]]=[[x,e_i],e_1]-[[x,e_1],e_i]=0$. Therefore, we have $[x,e_i]=0,\ \ i\geq3.$}}
Therefore, from
$$[x,[e_1,x]]-[[x,e_1],x]+[[x,x],e_1]=0$$
we get $$0-\gamma_{1,2}(e_2+\sum\limits_{k=3}^{t}\beta_ke_{k})+
\sum\limits_{k=2}^{t}\mu_{k}e_{k+1}=0$$
hence
 $$\gamma_{1,2}=0, \mu_{k}=0,\  2\le k\le t\ \mbox{and thus}\ [x,x]=0.$$
From the identity $Leib(x,e_i,x)=0,\ i\in\{1,2\}$ we obtain $[x,x]=[x,e_2]=0.$
%[x,[e_2,x]]-[[x,e_2],x]+[[x,x],e_2]=0,\Rightarrow \ [x,e_2+\sum\limits_{k=3}^t\beta_{k}e_{k}]-0+0=0,\ \Rightarrow\ \gamma_{2,1}e_1=0\Rightarrow \gamma_{2,1}=0;
We are left with $[e_i,x]=e_i+\sum\limits_{k=3}^t\beta_{k}e_{k+i-2},\ \ i\geq2$ and this completes the proof.
\end{proof}

The following proposition shows that two of the family of algebras obtained are not complete.
\begin{prop}
The derivations
\begin{itemize}
   \item {$d(e_i)=e_i,\ i\geq2$}\\
and
   \item {$d(e_i)=e_{i+2},\  i\geq2,$}
 \end{itemize}
 are outer derivations of the algebras $R_1(F,1,\beta)$ and $R_2(F,1,\beta)$,  respectively.
\end{prop}
\begin{proof}
The proof of the proposition is straightforward if one uses the tables of multiplications along with the derivation rule.
\end{proof}

\begin{thm}\label{t2}
An arbitary algebra of the family $R_2(F,1,\beta)$ is complete.
\end{thm}
\begin{proof} The fact that the center of $R_2(F,1,\beta)$ is trivial follows immediately from the table of multiplications. We prove that all derivations of $R_2(F,1,\beta)$ are inner. %the proof of this is similar to that of Theorem \ref{thm3_3}.
Note that $R_2(F,1,\beta)=F\oplus \mathbb{C}x$, where $Q=\mathbb{C}x$, and $\{e_1, e_2, x\}$ are generators of $R_2(F,1,\beta)$. %Since a derivation is completely defined by its value on generators, then it is sufficient to prove the existence $c\in R(F_2(\beta),1)$ such that $d(z)=ad_c(z)$ for $z\in \{e_1, e_2, x\}$.
Note also that for $k\in \mathbb{N}$ the quotient algebra $R_2(F,1,\beta)/F^k$ is isomorphic to the finite-dimensional Leibniz algebra with the table of multiplications (\ref{O'}). We make use the fact that all derivations of the algebra with the table of multiplications (\ref{O'}) are inner what was proven in \cite{O'ktam}.
%By similar assumption in the proof of Theorem \ref{thm3},
We set
$$d(e_1)=\sum\limits_{i=1}^{s}a_{i}e_i, \quad d(e_2)=\sum\limits_{i=1}^{s}b_{i}e_i, \quad
  d(x)=\sum\limits_{i=1}^{s}\gamma_{i}e_i+\gamma_{1,1}x+\gamma_{2,2}y.
$$

and put \quad $\bar{c}_k=\sum\limits_{i=1}^{k}\alpha_{k,i}e_i+\lambda_{k}x+F^k.$ From the equalities
$\overline{d(e_1)}=[\overline{e}_1,\bar{c}_k]$, \ \ $\overline{d(e_2)}=[\overline{e}_2,\bar{c}_k]$ and $\overline{d(x)}=[\overline{x},\bar{c}_k]$
 we get
$$\sum\limits_{i=1}^{s}a_{i}e_i-\lambda_ke_1, \quad \sum\limits_{i=1}^{s}b_{i}e_i-(\lambda_k+\mu_k)e_2-
\lambda_k\sum_{q=3}^{t}\beta_qe_q, \quad \sum\limits_{i=1}^{s}\gamma_{i}e_i+\gamma_{1,1}x+\alpha_{k,1}e_1\in F^k.$$
This implies
\begin{equation}\label{eq3_2}
a_1=\lambda_k, \quad b_2=\lambda_k+\mu_k, \quad \gamma_1=-\alpha_{k,1}.
\end{equation}

From \eqref{eq3_2} we conclude that $c_k=c_{k+1}$ for any $k\geq max\{s,t\}$, then setting $c:=c_k$, for the smallest $k$
obtain $d=ad_{c}.$
\end{proof}
By the following theorem we describe
 the family $R(F,2)$ of residually solvable Leibniz algebras whose maximal by inclusion pro-nilpotent ideal is $F$ and the subspace complementary to $F$ is maximal.

\begin{thm}\label{thm3} The algebra $R(F,2)$ admits a basis $\{x, y, e_{1}, e_{2}, \dots\}$ such that the table of multiplication of $R(F,2)$ on this basis has the following form
$$R(F,2,\beta):\left\{\begin{array}{ll}
[e_{i} ,e_{1}]=e_{i+1},\quad i\geq2,\\[1mm]
[e_1,x]=-[x,e_1]=e_1\\[1mm]
[e_i,x]=(i-2)e_{i}+\sum\limits_{k=2}^t\beta_{k}e_{k+i-2},\ \ i\geq2,\\[1mm]
[e_i,y]=e_{i},\ \  i\geq2,
\end{array}\right.$$
where $\beta=(\beta_2,\beta_3,\dots,\beta_t)\in\mathbb{C}^{t-1} \ \text{for \ some}\  t\in \mathbb{N}.$
\end{thm}
\begin{proof}
Due to Proposition \ref{prop01} we have
$$[e_1,x]=e_1,\quad [e_i,x]=(i-2)e_{i}+\sum\limits_{k=2}^t\beta_{k}e_{k+i-2},\quad
[e_i,y]=e_{i}+\sum\limits_{k=2}^t\beta_{k}'e_{k+i-2}, \quad i\geq2.$$

Let
$$\begin{array}{llll}
[x,e_i]=\sum\limits_{k=1}^{t}\gamma_{i,k}e_k,& &[y,e_i]=\sum\limits_{k=1}^{t}\eta_{i,k}e_k,\ \quad i\in\{1,2\},\\[1mm]
[x,y]=\sum\limits_{k=1}^{t}\nu_{k}e_k+C_{1}x+C_{2}y,& &
[y,x]=\sum\limits_{k=1}^{t}\nu_{k}'e_k+C_{1}'x+C_{2}'y,\\[1mm]
[x,x]=\sum\limits_{k=1}^{t}\mu_{k}e_k+D_{1}x+D_{2}y,& &
[y,y]=\sum\limits_{k=1}^{t}\mu_{k}'e_k+D_{1}'x+D_{2}'y.\\[1mm]
\end{array}$$

%$$[e_1,[x,y]]-[[e_1,x],y]+[[e_1,y],x]=0\ \Rightarrow\  C_1=0;$$
%$$[e_2,[x,y]]-[[e_2,x],y]+[[e_2,y],x]=0\ \Rightarrow\  C_2=0;$$
%
%$$[e_1,[y,x]]-[[e_1,y],x]+[[e_1,x],y]=0\ \Rightarrow\  C_1'=0;$$
%$$[e_2,[y,x]]-[[e_2,y],x]+[[e_2,x],y]=0\ \Rightarrow\  C_2'=0;$$
% From $$Leib(e_i,x,y)=Leib(e_i,y,x)=Leib(e_i,x,x)=Leib(e_i,y,y)=0,\  i\in\{1,2\}$$
% we obtain $$[x,y],\ [y,x],\ [x,x],\ [y,y]\in F.$$

The identities $Leib(e_i,Q,Q)=0,\quad  i\in\{1,2\}$ imply $[Q,Q]\in F.$
%$$[e_1,[x,e_1]]-[[e_1,x],e_1]+[[e_1,e_1]x]=0,\ \Rightarrow\
%[e_1,\sum\limits_{k=1}^{t}\gamma_{1,k}e_k]-[e_1,e_1]+0=0;$$
%$$[e_1,[x,e_2]]-[[e_1,x],e_2]+[[e_1,e_2]x]=0,\ \Rightarrow\
%[e_1,\sum\limits_{k=1}^{t}\gamma_{2,k}e_k]-[e_1,e_2]+0=0,\Rightarrow 0=0;$$
%And from $Leib(e_2,x,e_i)=Leib(e_2,y,e_i)=0,\ i\in\{1,2\}$, we get $\gamma_{1,1}=-1,\ \gamma_{2,1}=\eta_{1,1}=\eta_{2,1}=0.$
From $Leib(e_2,Q,e_i)=0,\ i\in\{1,2\}$ we obtain $\gamma_{1,1}=-1,\ \gamma_{2,1}=\eta_{1,1}=\eta_{2,1}=0.$
%$$[e_2,[x,e_1]]-[[e_2,x],e_1]+[[e_2,e_1],x]=0,\ \Rightarrow\
%[e_2,\sum\limits_{k=1}^{t}\gamma_{1,k}e_k]-
%[\sum\limits_{k=1}^t\beta_{k}e_{k+2},e_1]+e_{3}+\sum\limits_{k=1}^t\beta_{k}e_{k+3}=0,\Rightarrow $$
%$$\gamma_{1,1}e_3-\sum\limits_{k=1}^t\beta_{k}e_{k+3}+e_{3}+\sum\limits_{k=1}^t\beta_{k}e_{k+3}\Rightarrow\gamma_{1,1}=-1;$$
%
%$$[e_2,[x,e_2]]-[[e_2,x],e_2]+[[e_2,e_2],x]=0,\ \Rightarrow\
%[e_2,\sum\limits_{k=1}^{t}\gamma_{2,k}e_k]=0,\Rightarrow \gamma_{2,1}e_{3},\Rightarrow \ \gamma_{2,1}=0;$$
%$$[e_2,[y,e_1]]-[[e_2,y],e_1]+[[e_2,e_1],y]=0,\Rightarrow
%[e_2,\sum\limits_{k=1}^{t}\eta_{1,k}e_k]-[e_{2}+\sum\limits_{k=1}^t\beta_{k}'e_{k+2},e_1]+e_{3}+\sum\limits_{k=1}^t\beta_{k}'e_{k+3}=0,\Rightarrow$$
%$$\eta_{1,1}e_3=0,\Rightarrow\ \ \eta_{1,1}=0;$$
%
%$$[e_2,[y,e_2]]-[[e_2,y],e_2]+[[e_2,e_2],y]=0,\Rightarrow
%[e_2,\sum\limits_{k=1}^{t}\eta_{2,k}e_k]-0+0=0,\Rightarrow$$
%$$\eta_{2,1}e_3=0\ \Rightarrow\ \eta_{2,1}=0,\Rightarrow \ [y,e_1]=\sum\limits_{k=2}^{t}\eta_{1,k}e_k;$$
%By using $Leib(x,e_1,e_2)=Leib(y,e_1,e_2)=0$, we get $[Q,e_2]=0$. And by induction we can prove $[Q,e_i]=0\  i\geq 2$.
Use $Leib(Q,e_1,e_2)=0$ to get $[Q,e_2]=0$ which is the base of induction. Hence, applying the induction we obtain $[Q,e_i]=0$ for $i\geq 2$.

%$$[x,[e_1,e_2]]-[[x,e_1],e_2]+[[x,e_2],e_1]=0\Rightarrow 0-0+\sum\limits_{k=2}^{t}\gamma_{2,k}e_{k+1}=0,\Rightarrow \gamma_{2,k}=0,\ 2\le k\le t,\ \Rightarrow [x,e_2]=0;$$
%
%$$[x,e_3]=[[x,e_2],e_1]]-[[x,e_1],e_2]=0;$$
%
%And $[x,e_i]=0,\  i\geq 3$  can be proved by induction.

The equalities $[x,[x,x]]=0$ and $[x,[y,y]]=0$ give $[x,x]=\sum\limits_{k=2}^{t}\mu_{k}e_k$ and $[y,y]=\sum\limits_{k=2}^{t}\mu_{k}'e_k$, respectively.
%From $Z(x,x,x)=0$, we get $[x,x]=\sum\limits_{k=2}^{t}\mu_{k}e_k$.
%From $Z(x,y,y)=0$, we get $[y,y]=\sum\limits_{k=2}^{t}\mu_{k}'e_k$,
%
%$$[y,[e_1,e_2]]-[[y,e_1],e_2]+[[y,e_2],e_1]=0,\Rightarrow 0-0+\sum\limits_{k=2}^{t}\eta_{2,k}e_{k+1}=0,\ \Rightarrow\eta_{2,k}=0,\  2\le k\le t;$$
%Hence, $[y,e_2]=0.$ And we can prove $[y,e_i]=0, \ i\geq3,$ by induction.

The base change
 $$x'=x-\sum\limits_{k=2}^{t-1}\gamma_{1,k+1}e_{k},\ \ \text{and}\ \ y'=y-\sum\limits_{k=2}^{t-1}\eta_{1,k+1}e_{k}.$$
 enables us to write
$$[x,e_1]=-e_1+\gamma_{1,2}e_2,\ \quad [y,e_1]=\eta_{1,2}e_2.$$
From the equalities $Leib(x,e_1,x)=Leib(y,e_1,y)=0$ we derive $\gamma_{1,2}=\eta_{1,2}=\mu_{k}=\mu_{k}'=0, \ 2\le k\le t$ which leads to $[x,e_1]=-e_1$ and $[x,x]=[y,y]=[y,e_1]=0.$
%$$[x,[e_1,x]]-[[x,e_1],x]+[[x,x],e_1]=0,\Rightarrow
%[x,e_1]-[-e_1+\gamma_{1,2}e_2,x]+\sum\limits_{k=2}^{t}\mu_{k}e_{k+1}=0,\Rightarrow$$
%$$-e_1+\gamma_{1,2}e_2+e_1-\gamma_{1,2}\sum\limits_{k=1}^t\beta_{k}e_{k+2}+\sum\limits_{k=2}^{t}\mu_{k}e_{k+1}=0,\Rightarrow \gamma_{1,2}=0,\ \mu_{k}=0, \ 2\le k\le t,\Rightarrow [x,x]=0,\ [x,e_1]=-e_1;$$
%
%$$[y,[e_1,y]]-[[y,e_1],y]+[[y,y],e_1]=0,\Rightarrow
%0-\eta_{1,2}[e_2,y]+\sum\limits_{k=2}^{t}\mu_{k}'e_{k+1}=0,\Rightarrow$$
%$$\eta_{1,2}=0,\ \mu_{k}'=0, \ 2\le k\le t,\Rightarrow [y,y]=0,\ [y,e_1]=0;$$
%
%Hence,$\\[1mm]
%[x,e_1]=-e_1\\[1mm]
%%[x,e_i]=0,\ i\geq1,\\[1mm]
%%[y,e_i]=0,\ i\geq2,\\[1mm]
%[x,y]=\sum\limits_{k=1}^{t}\nu_{k}e_k,\\[1mm]
%[y,x]=\sum\limits_{k=1}^{t}\nu_{k}'e_k,\\[1mm]
%%[x,x]=0,\\[1mm]
%%[y,y]=0,\\[1mm]
%$
%
%and
%$\\[1mm]
%[e_1,x]=e_1,\\[1mm]
%[e_i,x]=(i-2)e_{i}+\sum\limits_{k=1}^t\beta_{k}e_{k+i},\ \ i\geq2,\\[1mm]
%%[e_1,y]=0,\\[1mm]
%[e_i,y]=e_{i}+\sum\limits_{k=1}^t\beta_{k}'e_{k+i},\ \ i\geq2.\\[1mm]$
The equalities $Leib(x,y,e_1)=Leib(y,x,e_1)=0$
%From $[[x,y],e_1]=[x,[y,e_1]]+[[x,e_1],y]=0$ and $[[y,x],e_1]=[y,[x,e_1]]+[[y,e_1],x]=0$
give $[x,y]=\nu_1e_1$ and $[y,x]=\nu_1'e_1$.
%Hence,$\\[1mm]
%[x,e_1]=-e_1\\[1mm]
%%[x,e_i]=0,\ i\geq1,\\[1mm]
%%[y,e_i]=0,\ i\geq2,\\[1mm]
%[x,y]=\nu_{1}e_1,\\[1mm]
%[y,x]=\nu_{1}'e_1,\\[1mm]
%%[x,x]=0,\\[1mm]
%%[y,y]=0,\\[1mm]
%$
%
%and
%$\\[1mm]
%[e_1,x]=e_1,\\[1mm]
%[e_i,x]=(i-2)e_{i}+\sum\limits_{k=1}^t\beta_{k}e_{k+i},\ \ i\geq2,\\[1mm]
%%[e_1,y]=0,\\[1mm]
%[e_i,y]=e_{i}+\sum\limits_{k=1}^t\beta_{k}'e_{k+i},\ \ i\geq2.\\[1mm]$
Use the base change $y'=y+\nu_1e_1$ to get
$[x,y]=0,\ \ [y,x]=(\nu_1'+\nu_1)e_1.$
The identity $Leib(x,y,x)=0$ gives $[y,x]=0.$

%
%$$[x,[y,x]]-[[x,y],x]+[[x,x],y]=0,\Rightarrow -(\nu_1'+\nu_1)e_1=0,\Rightarrow \nu_1'+\nu_1=0,\Rightarrow [y',x]=0;$$
%Hence,$\\[1mm]
%[x,e_1]=-e_1\\[1mm]
%[x,e_i]=0,\ i\geq1,\\[1mm]
%[y,e_i]=0,\ i\geq2,\\[1mm]
%[x,y]=0,\\[1mm]
%[y,x]=0,\\[1mm]
%[x,x]=0,\\[1mm]
%[y,y]=0,\\[1mm]
%[e_1,x]=e_1,\\[1mm]
%[e_i,x]=(i-2)e_{i}+\sum\limits_{k=1}^t\beta_{k}e_{k+i},\ \ i\geq2,\\[1mm]
%[e_1,y]=0,\\[1mm]
%[e_i,y]=e_{i}+\sum\limits_{k=1}^t\beta_{k}'e_{k+i},\ \ i\geq2.\\[1mm]$

Finally, we make use  $Leib(e_2,x,y)=0$ to get $\beta_{k}'=0, \ 2\le k\le t$. Thus we obtain the required table of multiplications of $R(F,2,\beta)$.
%$[e_2,[x,y]]-[[e_2,x],y]+[[e_2,y],x]=0,\Rightarrow [\sum\limits_{k=1}^t\beta_{k}e_{k+2},y]+[e_{2}+\sum\limits_{k=1}^t\beta_{k}'e_{k+2},x]=0,\Rightarrow$
%$$-\sum\limits_{k=1}^t\beta_k(e_{k+2}+\sum\limits_{s=1}^t\beta_{s}'e_{k+s+2})+\sum\limits_{k=1}^t\beta_{k}e_{k+2}+
%\sum\limits_{k=1}^t\beta_{k}'(ke_{k+2}+\sum\limits_{s=1}^t\beta_{s}e_{s+k+2})=0,\Rightarrow$$
%
%$$-\sum\limits_{k=1}^t\beta_k\sum\limits_{s=1}^t\beta_{s}'e_{k+s+2}+
%\sum\limits_{k=1}^tk\beta_{k}'e_{k+2}+\sum\limits_{k=1}^t\beta_{k}'\sum\limits_{s=1}^t\beta_{s}e_{s+k+2}=0,\Rightarrow
%\sum\limits_{k=1}^tk\beta_{k}'e_{k+2}=0,\Rightarrow \beta_{k}'=0, \ 1\le k\le t;$$
%
%Hence,
%$$\left\{\begin{array}{lll}
%[e_1,x]=-[x,e_1]=e_1,\\[1mm]
%[e_i,x]=(i-2)e_{i}+\sum\limits_{k=2}^t\beta_{k}e_{k+i-2},\ \ i\geq2,\\[1mm]
%[e_i,y]=e_{i},\ \  i\geq2.
%\end{array}\right.$$
\end{proof}

\begin{thm}\label{t3}
All Leibniz algebras $R_2(F,2,\beta)$ are complete .
\end{thm}

\begin{proof} The fact that the center trivial is obvious from the table of multiplications.
Now, we prove that all derivations are inner.
 %The proof is similar to that of Theorem \ref{thm3_3}.
Recall that $R(F,2,\beta)=F\oplus Q$, where $\{x,y\}$ is the basis of $Q$ and $\{e_1, e_2, x,y\}$ are generators of $R(F,2,\beta)$. Consider the quotient algebra $(R(F,2,\beta))_k=\overline{F}\oplus Q,$ where $\overline{F}=F/F^k$ for $k\in \mathbb{N}$. It is isomorphic to ($R(F,2,0))_k$.
 %and due to a result of the paper  {\color{green}{ (bu natija hech qayerda chiqarilmagan ekan. \cite{Ladra} bu ishda bu algebraning 2-tartibli cohomology group i trivialligi isbotlangan xolos. shuning uchun shu yerda :
 By using the derivation rule and the table of multiplications of $(R(F,2,0))_k$ it is not hard to see that all derivation of the algebra $(R(F,2,0))_k$ are inner. Therefore, all derivations of $(R(F,2,0))_k=\overline{F}\oplus Q,\  k\in \mathbb{N}$ also are inner.

Set
$$d(e_1)=\sum\limits_{i=1}^{s}a_{i}e_i, \quad d(e_2)=\sum\limits_{i=1}^{s}b_{i}e_i, \quad
  d(x)=\sum\limits_{i=1}^{s}\gamma_{i}e_i+\gamma_{1,1}x+\gamma_{2,2}y.
$$
If we choose $\bar{c}_k=\sum\limits_{i=1}^{k}\alpha_{k,i}e_i+\lambda_{k}x+\mu_{k}y+F^k$, %($\bar{c}_k=\sum\limits_{i=1}^{k}\alpha_{k,i}e_i+\lambda_{k}x+\mu_{k}y
%+F^k$)
then $\overline{d(e_1)}=[\overline{e}_1,\bar{c}_k], \quad \overline{d(e_2)}=[\overline{e}_2,\bar{c}_k], \quad \overline{d(x)}=[\overline{x},\bar{c}_k].$ Hence,
$$\begin{array}{lll}
\sum\limits_{i=1}^{s}a_{i}e_i-\lambda_ke_1,\quad \sum\limits_{i=1}^{s}b_{i}e_i-\lambda_k\sum\limits_{k=3}^{s}\beta_ke_k-\mu_ke_2,\quad
\sum\limits_{i=1}^{s}\gamma_{i}e_i+\gamma_{1,1}x+\gamma_{2,2}y+\alpha_{k,1}e_1\in F^k,&\quad
%\sum\limits_{i=1}^{s}\tau_{i}e_i+\tau_{1,1}x+\tau_{2,2}y\in F^k.
\end{array}$$
This implies
\begin{equation}\label{eq1}\begin{array}{ll}
a_1=\lambda_k,\quad b_2=\mu_k,\quad\gamma_1=-\alpha_{k,1}.
\end{array}
\end{equation}
Thus,\quad
$c_k=-\gamma_1e_1+a_1x+b_2y.$ From \eqref{eq1} we conclude that $c_k=c_{k+1}$ for any $k\geq \max\{s,t\}$. Let $l$ be the smallest $k$ satisfying this condition. Then
% {\color{green}{and for such kind of the smallest $k$, by setting
setting $c:=c_l$ we get $d=ad_{c}$.
\end{proof}%

Now we treat the second cohomology groups of $R(F,2,\beta)$.

%Bunda biz ikkinchi kotsikl va kochegaralarining ayirmasini nol ekanligini ko'rsatganmiz.
\begin{thm}\label{thm4} The second cohomology groups $\text{HL}^2(R(F,2,\beta),R(F,2,\beta))$ of the algebras $R(F,2,\beta)$ are trivial. %i.e.,
%$H^2(R(F(\beta),2),R(F(\beta),2))=0.$
\end{thm}
\begin{proof} We follow the same strategy as in Theorem \ref{thm5}, i.e., prove that any cocycle is generated by a linear function $f\in \text{Hom}(R(F,2,\beta),R(F,2,\beta))$.

Let $\varphi\in Z^2(R(F,2,\beta),R(F,2,\beta))$. We expand $\varphi$ by the basis $\{x, y, e_1, e_2,...\}$ as follows
$$\varphi(e_i,e_j)=\sum\limits_{k=1}^{p(i,j)}a_k^{i,j}e_k+a_{1,1}^{i,j}x+a_{2,2}^{i,j}y, \ \ \varphi(e_i,x)=\sum\limits_{k=1}^{s(i)}b_k^ie_k+b_{1,1}^{i}x+b_{2,2}^{i}y, \ \ \varphi(x,e_i)=\sum\limits_{k=1}^{s(i)}d_k^ie_k+d_{1,1}^{i}x+d_{2,2}^{i}y,$$
$$\varphi(e_i,y)=\sum\limits_{k=1}^{s(i)}c_k^ie_k+c_{1,1}^{i}x+c_{2,2}^{i}y, \ \
\varphi(y,e_i)=\sum\limits_{k=1}^{s(i)}f_k^ie_k+f_{1,1}^{i}x+f_{2,2}^{i}y,\ \ \varphi(x,y)=\sum\limits_{k=1}^{s}g_k^ie_k+g_{1,1}^{i}x+g_{2,2}^{i}y,$$
$$\varphi(y,x)=\sum\limits_{k=1}^{s}u_k^ie_k+u_{1,1}^{i}x+u_{2,2}^{i}y,\ \ \varphi(x,x)=\sum\limits_{k=1}^{s}v_k^ie_k+v_{1,1}^{i}x+v_{2,2}^{i}y, \ \
\varphi(y,y)=\sum\limits_{k=1}^{s}r_k^ie_k+r_{1,1}^{i}x+r_{2,2}^{i}y,$$
where $i,j \in \mathbb{N}$.

We claim that $\varphi$ is generated by the linear function $f\in \text{Hom}(R(F,2),R(F,2))$ given below
\begin{equation}\label{34}
\begin{array}{lll}
f(e_{1})=-\alpha_1^1e_1-\sum\limits_{k=2}^{p_{1,1}}a_{k+1}^{1,1}e_k+b_{1,1}^1x+b_{2,2}^1y,& & f(e_{2})=d_1^2e_1+\sum\limits_{k=2}^{s_2}\alpha_k^2e_k+\alpha_{1,1}^2x+\alpha_{2,2}^2y,\\[1mm]
f(x)=\tau_1e_1-\sum\limits_{k=2}^{q}g_ke_k-b_1^1x+\tau_{2,2}y,& &
f(y)=-g_1e_1-\sum\limits_{k=2}^{s}r_ke_k-c_1^1x+\gamma_{2,2}y,\\[1mm]
f(e_{i})=d_1^ie_1+\sum\limits_{k=2}^{s_i}\alpha_k^ie_k+a_{1,1}^{i-1,1}x+a_{2,2}^{i-1,1}y,&& i\geq3,
\end{array}
\end{equation}

that is to prove that the cocycle $\chi=\varphi-\psi\in Z^2(R(F,2),R(F,2))$, where $\psi(x,y)=f([x,y])-[f(x),y]-[x,f(y)]$
is trivial. Expend the cocycle $\chi$ by the basis $\{x, y, e_1, e_2,...\}$

$\begin{array}{lll}
\chi(e_1,e_1)=a_1^{1,1}e_1+a_2^{1,1}e_2+a_{1,1}^{1,1}x+a_{2,2}^{1,1}y,\quad&\chi(e_i,y)=\sum\limits_{k=1}^{s(i)}c_k^ie_k+c_{1,1}^{i}x+c_{2,2}^{i}y,&i\geq3,\\[1mm]
\chi(e_1,e_i)=\sum\limits_{k=1}^{p(1,i)}a_k^{1,i}e_k+a_{1,1}^{1,i}x+a_{2,2}^{1,i}y,& \chi(e_i,e_1)=\sum\limits_{k=1}^{p(i,1)}a_k^{i,1}e_k, \quad i\geq2,\\[1mm]
\chi(e_i,e_j)=\sum\limits_{k=1}^{p(i,j)}a_k^{i,j}+a_{1,1}^{i,j}x+a_{2,2}^{i,j}y,& \chi(e_i,x)=\sum\limits_{k=1}^{s_i}b_k^ie_k+b_{1,1}^{i}x+b_{2,2}^{i}y,& i,j\geq2,\\[1mm]
\chi(x,e_i)=\sum\limits_{k=1}^{s_i}d_k^ie_k+d_{1,1}^{i}x+d_{2,2}^{i}y,& \chi(y,e_i)=\sum\limits_{k=1}^{s(i)}f_k^ie_k+f_{1,1}^{i}x+f_{2,2}^{i}y,& i\geq1,\\[1mm]
\chi(e_1,y)=\sum\limits_{k=2}^{s_1}c_k^1e_k+c_{1,1}^{1}x+c_{2,2}^{1}y,&\chi(e_1,x)=\sum\limits_{k=2}^{s_1}b_k^1e_k,\\[1mm]
\chi(y,x)=\sum\limits_{k=1}^{s}u_k^ie_k+u_{1,1}^{i}x+u_{2,2}^{i}y,&\chi(x,y)=g_{1,1}x+g_{2,2}y,\\[1mm]
\chi(x,x)=\sum\limits_{k=1}^{s}v_k^ie_k+v_{1,1}^{i}x+v_{2,2}^{i}y,&\chi(y,y)=r_1e_1+r_{1,1}^{i}x+r_{2,2}^{i}y,\\[1mm]
\chi(e_2,y)=c_1^2e_1+\sum\limits_{k=3}^{s(i)}c_k^ie_k+c_{1,1}^{i}x+c_{2,2}^{i}y.\\[1mm]
\end{array}$

Now, we impose the cocycle identities Z = 0 (see \ref{Coc}) to $\chi$ and derive a set of constraints for the coefficients as follows.

\begin{center}
    \begin{tabular}{llllll}
       \qquad\quad 2-cocyle identity & &\qquad\qquad\qquad Constraints\\
        \hline \hline
\\
        $Z(e_1,e_1,e_1)=0,$ &\quad $\Rightarrow $\quad & $\left\{\begin{array}{ll}
                                                                 a_{1,1}^{1,1}=0,\\[1mm]

                                                                 \end{array}\right.$\\
        $Z(e_2,e_1,e_1)=0, $ &\quad $\Rightarrow $\quad & $\left\{\begin{array}{ll}
                                                                 a_{1}^{1,1}=a_{2,2}^{1,1}=0,\\[1mm]
                                                                 \end{array}\right.$\\
        $Z(e_1,e_1,y)=0,\quad i\geq2, $ &\quad $\Rightarrow $\quad & $\left\{\begin{array}{ll}
                                                                 a_2^{1,1}=0,\ c_k^1=0,\ 2\le k\le s_1,\\[1mm]
                                                                 \end{array}\right.$\\
        $Z(e_1,e_1,x)=0,\quad i\geq2, $ &\quad $\Rightarrow$ \quad &$\left\{\begin{array}{ll}
                                                                  b_k^1=0,\ 2\le k\le s_1,\\[1mm]
                                                                  \end{array}\right.$\\
$Z(e_1,x,e_1)=0,$ &\quad $\Rightarrow $\quad & $\left\{\begin{array}{ll}
                                                                 d_{1,1}^1=0,\\[1mm]

                                                                 \end{array}\right.$\\
$Z(x,e_1,y)=0,$ &\quad $\Rightarrow $\quad & $\left\{\begin{array}{ll}
                                                                 g_{1,1}^{1}=c_{1,1}^1=c_{2,2}^1=0,\\[1mm]
                                                                 d_k^1=0, \ \ 2\le k\le s_1,\\[1mm]
                                                                 \end{array}\right.$\\

$Z(e_1,e_1,e_i)=0,\quad i\geq2,$ &\quad $\Rightarrow $\quad & $\left\{\begin{array}{ll}
                                                                 a_k^{1,i}=0, \ i\geq2, \ 2\le k\le p(1,i),\\[1mm]
                                                                 \end{array}\right.$\\
$Z(e_1,e_i,y)=0,\quad i\geq3,$ &\quad $\Rightarrow $\quad & $\left\{\begin{array}{ll}
                                                                 a_1^{1,i}=-c_{1,1}^i,\ a_{1,1}^{1,i}=a_{2,2}^{1,i}=0, \ i\geq3,\\[1mm]
                                                                 \end{array}\right.$\\

%\[[e_1,\psi(e_1,e_1)]=a_{1,1}^{1,1}e_1=0,\Rightarrow \ a_{1,1}^{1,1}=0;\]
%
%\[[e_2,\psi(e_1,e_1)]=a_{1}^{1,1}e_3+a_{2,2}^{1,1}e_2=0,\Rightarrow \ a_{1}^{1,1}=a_{2,2}^{1,1}=0;\]

%\[[e_1,\psi(e_1,y)]-[\psi(e_1,e_1),y]+[\psi(e_1,y),e_1]=0\Rightarrow
%[c_{1,1}^{1}e_1-a_2^{1,1}e_2+\sum\limits_{k=2}^{s_1}c_k^1e_{k+1}-c_{1,1}^{1}e_1=0,\Rightarrow\]
%\[a_2^{1,1}=0,\ c_k^1=0,\ 2\le k\le s_1,\Rightarrow \psi(e_1,e_1)=0,\ \psi(e_1,y)=c_{1,1}^1x+c_{2,2}^1y;\]

%\[[e_1,\psi(e_1,x)]-[\psi(e_1,e_1),x]+[\psi(e_1,x),e_1]+2\psi(e_1,e_1)=0,\Rightarrow \]
%\[\sum\limits_{k=2}^{s_1}b_k^1e_{k+1}=0,\Rightarrow b_k^1=0,\ 2\le k\le s_1,\ \Rightarrow \psi(e_1,x)=0; \]
%
%\[[e_1,\psi(x,e_1)]-[\psi(e_1,x),e_1]+[\psi(e_1,e_1),x]-2\psi(e_1,e_1)=0,\Rightarrow \psi(x,e_1)=\sum\limits_{k=1}^{s_1}d_k^1e_k+d_{2,2}^{1}y;\]

%\[[x,\psi(e_1,y)]-[\psi(x,e_1),y]+[\psi(x,y),e_1]+\psi(e_1,y)=0,\Rightarrow
%-\sum\limits_{k=2}^{s_1}d_k^1e_k-g_{1,1}^{1}e_1+c_{1,1}^1x+c_{2,2}^1y=0,\ \Rightarrow\]
% \[g_{1,1}^{1}=c_{1,1}^1=c_{2,2}^1=0,\ d_k^1=0,\ 2\le k\le s_1,\Rightarrow \psi(e_1,y)=0,\quad \psi(x,e_1)=d_1^1e_1+d_{1,1}^1x+d_{2,2}^1y;\]
%\[Z(e_1,e_1,e_i)=0\Rightarrow \psi(e_1,e_i)=a_1^{1,i}e_1+a_{1,1}^{1,i}x+a_{2,2}^{1,i}y;\]
%\[Z(e_1,e_i,y)=0\Rightarrow \psi(e_1,e_i)=-c_{1,1}^ie_1;\]

        $Z(x,e_1,e_i)=0,\quad i\geq2,$ &\quad $\Rightarrow $\quad & $\left\{\begin{array}{ll}
                                                                 d_{1,1}^{i}=d_k^{i}=0,\ 2\le k\le s_i,\ i\geq2,\\[1mm]

                                                                 \end{array}\right.$\\
        $Z(x,e_i,e_1)=0, $ &\quad $\Rightarrow $\quad & $\left\{\begin{array}{ll}
                                                                 a_{1}^{2,1}=c_{1,1}^2=0,\ a_1^{i,1}=c_{1,1}^i,\\[1mm]
                                                                 d_{2,2}^{i}=0,\ i\geq3,\\[1mm]
                                                                 \end{array}\right.$\\
        $Z(y,y,e_i)=0,\quad i\geq1, $ &\quad $\Rightarrow $\quad & $\left\{\begin{array}{ll}
                                                                 r_{1,1}=0,\\[1mm]
                                                                 \end{array}\right.$\\
        $Z(y,e_i,y)=0,\quad i\geq2, $ &\quad $\Rightarrow$ \quad &$\left\{\begin{array}{ll}
                                                                  f_1^i=f_{1,1}^i=f_{2,2}^i=0,\ i\geq2,\\[1mm]
                                                                  \end{array}\right.$\\
        $Z(e_1,x,x)=0,$ &\quad $\Rightarrow $\quad & $\left\{\begin{array}{ll}
                                                                 v_{1,1}=0,\\[1mm]
                                                                 \end{array}\right.$\\
        $Z(e_2,x,x)=0,$ &\quad $\Rightarrow $\quad & $\left\{\begin{array}{ll}
                                                                 v_{1}=v_{2,2}=0,\\[1mm]
                                                                 \end{array}\right.$\\
        $Z(x,x,e_1)=0,$ &\quad $\Rightarrow $\quad & $\left\{\begin{array}{ll}
                                                                 d_1^1=d_{1,1}^1=d_{2,2}=v_k=0,\ 2\le k\le s,\\[1mm]
                                                                 \end{array}\right.$\\
        $Z(e_2,y,y)=0, $ &\quad $\Rightarrow $\quad & $\left\{\begin{array}{ll}
                                                                 r_1=r_{2,2}=0,\\[1mm]
                                                                 \end{array}\right.$\\
        $Z(y,e_1,x)=0,\quad i\geq2,$ &\quad $\Rightarrow $\quad & $\left\{\begin{array}{ll}
                                                                 u_{1,1}=u_k=0,\ 2\le k\le s,\\[1mm]
                                                                 \end{array}\right.$\\
        $Z(e_i,y,e_j)=0,\quad i,j\geq2,$ &\quad $\Rightarrow $\quad & $\left\{\begin{array}{ll}
                                                                 a_1^{i,j}=a_{1,1}^{i,j}=a_{2,2}^{i,j}=0,\ i,j\geq2,\\[1mm]
                                                                 \end{array}\right.$\\

        $Z(e_1,x,y)=0,\quad i\geq2,$ &\quad $\Rightarrow $\quad & $\left\{\begin{array}{ll}
                                                                 g_{1,1}=0,\\[1mm]
                                                                 \end{array}\right.$\\
        $Z(x,e_i,y)=0,\quad i\geq2,$ &\quad $\Rightarrow $\quad & $\left\{\begin{array}{ll}
                                                                c_1^i=d_{2,2}^i=0,\ i\geq2,\\[1mm]
                                                                 \end{array}\right.$\\

        $Z(e_1,e_2,y)=0, $ &\quad $\Rightarrow $\quad & $\left\{\begin{array}{ll}
                                                                 a_{1,1}^2=a_{2,2}^2=c_{1,1}^2=0,\\[1mm]
                                                                 \end{array}\right.$\\
             \end{tabular}
\end{center}

\begin{center}
    \begin{tabular}{llllll}
       \qquad\quad 2-cocyle identity & &\qquad\qquad\qquad Constraints\\
        \hline \hline
\\
        $Z(e_2,e_1,y)=0, $ &\quad $\Rightarrow$ \quad &$\left\{\begin{array}{ll}
                                                                  c_2^{3}=0,\ c_k^i=c_{k-1}^2, 3\le k\le s_2+1,\\[1mm]
                                                                  \end{array}\right.$\\
     $Z(e_i,e_j,y)=0,\quad i\geq2,$ &\quad $\Rightarrow $\quad & $\left\{\begin{array}{ll}
                                                                 a_k^{i,j}=0, \ 2\le k\le p_{i,j}, \ i,j\geq2,\\[1mm]
                                                                 \end{array}\right.$\\
        $Z(e_i,x,y)=0,\ i\geq2,$ &\quad $\Rightarrow $\quad & $\left\{\begin{array}{ll}
                                                                c_k^2=0,\ 3\le k\le s_2,\\[1mm]
                                                                b_1^i=b_{1,1}^i=b_{2,2}^i=0,\ i\geq2,\\[1mm]
                                                                 \end{array}\right.$\\
        $Z(e_2,y,x)=0, $ &\quad $\Rightarrow $\quad & $\left\{\begin{array}{ll}
                                                                 u_1=u_{2,2}=0,\\[1mm]
                                                                 \end{array}\right.$\\

             \end{tabular}
\end{center}

As a result we get
$$\chi(e_i,e_1)=\sum\limits_{k=2}^{p_{i,1}}a_k^{i,1}e_{k},\quad \chi(e_i,x)=\sum\limits_{k=2}^{s_i}b_k^ie_{k},\ i\geq2.$$

Observe that we have not imposed any constraints for the components {{$\alpha_{k}^i,\ 2\leq i,\ 2\leq k\leq s_i$}}  of $f(e_i)$ in (\ref{34}) yet. Now choosing the coefficients $\alpha_{k}^i,\ 2\leq i,\ 2\leq k\leq s_i$ in (\ref{34}) we attain $\chi(e_i,e_1)=0$ for $i\geq 2$. Since, $\chi(e_i,x)$ are expressed via $\chi(e_i,e_1)$, from
%Biz umumiylikka ta'sir etmagan holda $\alpha_{k}^i,\ 2\leq i,\ 2\leq k\leq s_i$ koeffitsientlarning hisobiga $\psi(e_i,x)=0,\ i\geq 2$ ni olamiz.
 $Z(e_i,e_1,x)=0,\ i\geq 2$, we conclude that $\chi(e_i,x)=0,\ i\geq2$.
This completes the proof.
%
%$\chi(e_i,e_1)=(\varphi-\psi)(e_i,e_1)=\sum\limits_{k=2}^{p_{i,1}}a_k^{i,1}e_k-[f(e_i),e_1]
%-[e_i,f(e_1)]+f(e_{i+1})=0$
%
%$f(e_i)=d_1^i+\sum\limits_{k=2}^{s_i}\alpha_k^ie_k+a_{1,1}^{i-1,1}x+a_{2,2}^{i-1.1}$
\end{proof}

\textbf{Acknowledgements}. The authors are grateful to the referee for valuable
comments and remarks.

\end{document}